\documentclass[11pt,a4paper]{amsart}
\usepackage[left=25mm,top=30mm,bottom=25mm,right=25mm]{geometry}
\usepackage{mathtools, amsmath, amssymb, amsfonts, amsthm}
\usepackage[all]{xy}
\usepackage{bm}
\usepackage{mathrsfs}
\usepackage{dsfont}
\usepackage{tensor}
\usepackage{stmaryrd}
\usepackage[vcentermath]{youngtab}
\usepackage[utf8]{inputenc}
\usepackage{lmodern}
\usepackage[english]{babel}
\usepackage{graphicx}
\usepackage{wrapfig}
\usepackage[pdfdisplaydoctitle=true,
      colorlinks=true,
      urlcolor=blue,
      citecolor=blue,
      linkcolor=blue,
      pdfstartview=FitH,
      pdfpagemode=UseNone,
      bookmarksnumbered=true,
      unicode=true]{hyperref}
\usepackage{hyperxmp}
\usepackage{cmap}
\usepackage{tikz}

\newtheorem{thm}{Theorem}[section]

\newtheorem{prop}[thm]{Proposition}

\theoremstyle{definition}
\newtheorem{rem}[thm]{Remark}

\newtheorem{defn}[thm]{Definition}

\newtheorem{exam}[thm]{Example}

\numberwithin{equation}{section}

\newcommand{\RR}{\mathbb{R}}                

\newcommand{\set}[1]{\left\{#1\right\}}

\newcommand{\espace}{\mathcal{E}}           

\newcommand{\Cinf}{\mathrm{C}^{\infty}}

\newcommand{\vol}{\mathrm{vol}}     
\newcommand{\Scal}{s}               

\newcommand{\bA}{\mathbf{A}}
\newcommand{\bB}{\mathbf{B}}

\newcommand{\bD}{\mathbf{D}}
\newcommand{\bE}{\mathbf{E}}

\newcommand{\bH}{\mathbf{H}}

\newcommand{\bJ}{\mathbf{J}}
\newcommand{\bK}{\mathbf{K}}
\newcommand{\bL}{\mathbf{L}}

\newcommand{\bP}{\mathbf{P}}

\newcommand{\bT}{\mathbf{T}}

\newcommand{\bW}{\mathbf{W}}
\newcommand{\bX}{\mathbf{X}}

\newcommand{\ba}{\mathbf{a}}
\newcommand{\bb}{\mathbf{b}}

\newcommand{\bq}{\mathbf{q}}

\newcommand{\ee}{\bm{e}}

\newcommand{\xx}{\bm{x}}

\newcommand{\bOmega}{\bm{\Omega}}

\newcommand{\balpha}{\bm{\alpha}}
\newcommand{\bbbeta}{\bm{\beta}}
\newcommand{\bbeta}{\bm{\eta}}
\newcommand{\bphi}{\bm{\phi}}
\newcommand{\bepsilon}{\bm{\varepsilon}}

\newcommand{\bsigma}{\bm{\sigma}}

\newcommand{\bxi}{\bm{\xi}}

\newcommand{\id}{\mathrm{id}}

\newcommand{\rd}{\mathrm{d}}

\newcommand{\dd}[2]{\frac{\mathrm{d}{#1}}{\mathrm{d}{#2}}}

\newcommand{\pdi}[1]{\partial_{#1}}
\newcommand{\pdij}[1]{\partial^{2}_{#1}}

\DeclareMathOperator{\Lie}{\mathcal{L}} %
\DeclareMathOperator{\tr}{tr} %
\DeclareMathOperator{\Inc}{\mathbf{Inc}} 
\DeclareMathOperator{\grad}{\mathbf{grad}} %
\DeclareMathOperator{\dive}{\mathbf{div}} %
\DeclareMathOperator{\rot}{\mathbf{curl}} %

\hypersetup{
    pdftitle={The linear Elasticity complex: a natural formulation},
    pdfauthor={Romain Lloria and Boris Kolev},
    pdfsubject={2020 Mathematics Subject Classification: 74B05, 58J10, 18G35, 53Z30, 74A15},
    pdfkeywords={Elasticity complex, Saint--Venant compatibility condition, Cesàro--Volterra path integral, Homotopy operator, Dubois--Violette--Henneaux Generalized complex, Airy potential, Beltrami Stress Functions},
    pdflang=en
}
\begin{document}

\title{The linear Elasticity complex: a natural formulation}%

\author{R. Lloria}
\address[Romain Lloria]{Université Paris-Saclay, CentraleSupélec, ENS Paris-Saclay,  CNRS, LMPS - Laboratoire de Mécanique Paris-Saclay, 91190, Gif-sur-Yvette, France}
\email{romain.lloria@ens-paris-saclay.fr}

\author{B. Kolev}
\address[Boris Kolev]{Université Paris-Saclay, CentraleSupélec, ENS Paris-Saclay,  CNRS, LMPS - La\-bo\-ra\-toi\-re de Mécanique Paris-Saclay, 91190, Gif-sur-Yvette, France}
\email{boris.kolev@ens-paris-saclay.fr}

\date{April 24, 2026}%


\begin{abstract}
  We reformulate the Elasticity complex and Saint–Venant's compatibility condition using the generalized differential complex of Dubois–-Violette--Henneaux. This is just a slight and natural modification of the de Rham complex to take account of the index symmetry of the tensors involved. An integrating formula to recover the displacement from the strain and similar to the Poincaré formula is provided. Finally, a Hodge star operator and a dual complex is introduced, which allows to recover stress potentials in dimensions 2 and 3.
\end{abstract}

\maketitle

\section*{Introduction}

An infinitesimal deformation is said to be compatible if a displacement vector field generates the symmetric covariant second-order tensor field of infinitesimal deformation. The latter is then the symmetric part of the Jacobian with respect to the displacement vector field. In this case, two problems arise.

The first consists of finding a compatibility condition. It would inform us about the existence of a displacement vector field generating the symmetric covariant second-order tensor field of compatible infinitesimal deformation. Barré de Saint-Venant was the first one to establish a necessary compatibility condition \cite{dSai1855}. He noticed that the symmetry of the gradient was not enough to guarantee the existence of a displacement vector field, because there are more components in the symmetric covariant second-order tensor field of infinitesimal deformation than in the displacement vector field. He derives twice the equation defining the symmetric covariant second-order tensor field of infinitesimal deformation and finds combinations that eliminate the displacement vector field. He obtained the necessary relations to have a solution for the inverse equations. They define a necessary compatibility condition of Saint-Venant by cancellation of a covariant fourth-order tensor field called the \emph{Saint-Venant tensor}.
Later, Cesàro \cite{Ces1906}, Volterra \cite{Vol1906}, and Love \cite{Lov1944} obtain the same compatibility condition. The first two use the fundamental lemma of integral calculus and Stokes' theorem. Love rewrites the displacement vector field using the Helmholtz-Stokes decomposition.

The second problem consists of finding an explicit formula for the displacement vector field solution of the problem, given a compatible infinitesimal strain. \textit{A priori}, nothing guarantees the integrability of the previous differential equations. Beltrami showed in 1886 that these conditions were locally sufficient \cite{ACG2006}. Some topological conditions on the domain are required to obtain global solutions \cite{Ces1906,Vol1906,Car1966}, like in de Rham cohomology. The problem for non simply connected domains was investigated in \cite{Y2013}.

In 1999, Eastwood \cite{Eas1999,Eas2000} used the general \emph{Bernstein-Gelfand-Gelfand construction} (BGG) to derive the Elasticity complex from the de Rham complex with value in the Lie algebra of the Euclidean group. To do this, he constructs commutative diagrams joining the Elasticity complex to the de Rham complex. The fundamental property $d^2=0$ of the Rham complex transfers to this Elasticity complex. The compatibility condition is recast as the vanishing of a symmetric second-order tensor field called the \emph{incompatibility tensor} $\Inc \bepsilon = 0$. Using this approach, a Poincaré path integral formula (an integrator) for elasticity was derived in \cite{SKE2020}.

These integrators allow \cite{CCG2007, CGM2007} to recover the displacement vector field from the strain tensor field as the major variable in the intrinsic elasticity model. Generalizations of the theory under weaker regularity \cite{CGM2009, CGM2010, CH2024} have also been considered. In \cite{Arnold2007}, by mimicking this construction in the discrete case, they derive new mixed finite elements for elasticity in a systematic manner from known discretizations of the de Rham complex. These elements appear to be simpler than the ones previously derived. The BGG approach has also applications for the finite element method \cite{AFW2006,Fal2008,AAA2023,Hu2024}.

In this paper, we propose a new approach of the Elasticity complex based on Dubois--Violette and Henneaux \cite{DH1999, DH2002} \emph{generalized complexes}. We consider it as a \emph{natural} approach since it corresponds just to a straightforward generalization of the de Rham complex for other types of index symmetries. No need to introduce additional structures and isomorphisms to relate the de Rham complex to the Elasticity complex like in the BGG approach. Moreover, \emph{Poincaré like integrators} are also provided in our framework. An extended definition of the exterior derivative as well as an extended version of the \emph{Hodge star operator} lead naturally to the definition of the Saint--Venant tensor, the $\Inc$ tensor and a duality which allows also for a new geometric interpretation of Airy (2D) and Beltrami (3D) stress potentials \cite{Car1966,Gur1963,GP2004,Sad2009,GP2015}.

\subsection*{Outline}

In \autoref{sec:Cesaro-Voltera}, we formulate in modern language the Cesàro-Volterra path integral formula, which provides a compatibility condition for a strain $\bepsilon$ to derive from a displacement field $\bxi$. This condition involves a fourth order tensor: the Saint-Venant tensor. In \autoref{sec:Ricci}, we explain why the vanishing of the fourth order Saint--Venant tensor is equivalent to the vanishing of a symmetric second-order tensor, the incompatibility tensor and that this is specific to the dimension 3. In \autoref{sec:deRham-complex}, we introduce the de Rham complex to familiarize the reader with differential complexes and illustrate it through the theory of Electromagnetism. The Elasticity complex is derived in \autoref{subsec:elasticity-complex} as a special case of Dubois–Violette theory of generalized differential complexes. An homotopy formula is provided and allows us to recover the Saint-Venant compatibility condition and the Cesàro–Volterra integrator. Finally, in \autoref{sec:stress-potentials}, we propose to interpret the Airy (2D) and Beltrami (3D) stress potentials using the dual of the Elasticity complex and an extended version of the Hodge star operator.

\subsection*{Notations}

We denote by $(\ee_{1},\ee_{2},\ee_{3})$ the canonical basis of $\RR^3$, and by $\bq$, the canonical Euclidean metric $\RR^3$ with components $(\delta_{ij})$, the Kronecker symbols, in this basis, where the latin indices $i, j, k, \ldots$ vary between $1$ and $3$.

We assume infinitesimal deformations. We denote by $\bxi$ the displacement vector field and by $\bxi^{\flat}=\bq \bxi=(\xi_{i})$, the corresponding displacement covector field.
The covariant derivative of the latter, $\nabla\bxi^{\flat}$, is decomposed into symmetric and skew-symmetric parts, respectively equal to the infinitesimal strain tensor field
\begin{equation}\label{def:strain}
  \bepsilon:=\frac{1}{2}\left(\nabla\bxi^\flat+(\nabla\bxi^\flat)^{\star}\right),
  \qquad
  \left( \; \emph{i.e.}, \;\;
  \varepsilon_{ij}=\frac{1}{2}\left(\partial_{j}\xi_{i}+\partial_{i} \xi_{j}\right)
  \;
  \right),
\end{equation}
and the infinitesimal rotation tensor field (or spin tensor)
\begin{equation}\label{def:spin}
  \bOmega:=\frac{1}{2}\left(\nabla\bxi^\flat-(\nabla\bxi^\flat)^{\star}\right),
  \qquad
  \left( \; \emph{i.e.}, \;\;
  \Omega_{ij}=\frac{1}{2}\left(\partial_{j}\xi_{i}-\partial_{i} \xi_{j}\right)
  \;
  \right),
\end{equation}
where $L^{\star}$ means the dual operator (or dual transpose) of a linear operator $L$. We introduce also the mixed version of these tensors
\begin{equation*}
  \widehat{\bepsilon} := \frac{1}{2}\left(\nabla\bxi + (\nabla\bxi)^{t}\right) \quad \text{and} \quad \widehat\bOmega := \frac{1}{2}\left(\nabla\bxi-(\nabla\bxi)^{t}\right),
\end{equation*}
where $L^{t} := \bq^{-1}L^{\star}\bq$ is the metric transpose of $L$.

We shall also introduced the curl of a second-order mixed tensor field $\bL = (\tensor{L}{^i_{j}})$ on $\RR^3$ as in \cite{Sal2005,SKE2020,Gar2022}. We define, in components, the \emph{column curl} $\rot^c \bL$ by
\begin{equation}\label{eq:rot-c}
  {(\rot^c L)^{i}}_{j} := \tensor{\epsilon}{^{ip}_{q}} \partial_{p}\tensor{L}{^{q}_{j}},
\end{equation}
and the \emph{row curl} $\rot^r \bL$ by
\begin{equation}\label{eq:rot-r}
  {(\rot^r L)^{i}}_{j} := \tensor{\epsilon}{_{j}^{pq}} \partial_{p}\tensor{L}{^{i}_{q}}.
\end{equation}
Here, the Levi-Civita symbol is defined by $\epsilon_{ijk}:=\det(e_i,e_j,e_{k})$. The operators $\rot^c \bL$ and $\rot^r \bL$ can be defined intrinsically by a contraction between $\nabla \bL$ and the Riemannian volume form $\vol_{\bq}$ \cite{Gar2022}. One can check that
\begin{equation*}
  \rot^r \bL = (\rot^c \bL^{t})^{t},
\end{equation*}
and in particular that $\rot^r \widehat{\bepsilon} = (\rot^c \widehat{\bepsilon})^{t}$.

Finally, we recall the linear isomorphism $j$ between $\mathbb{R}^3$ and the space of skew-symmetric endomorphisms
\begin{equation}\label{def:j}
  j : \omega \in \mathbb{R}^3 \longmapsto \widehat{\Omega} = \omega \times \cdot,
\end{equation}
which is given in the canonical basis $(\ee_{i})$ by the formulas
\begin{equation*}
  \tensor{\widehat{\Omega}}{^i_{j}} = -\tensor{\epsilon}{^i_{jk}}\omega^{k}, \quad \text{and} \quad \omega^{i} = - \frac{1}{2} \tensor{\epsilon}{^{i}_{j}^{k}} \tensor{\widehat{\Omega}}{^j_{k}}.
\end{equation*}

\section{Cesàro-Volterra path integral formula}
\label{sec:Cesaro-Voltera}

In 1906, Volterra \cite{Vol1906} (see also~\cite{VV1960}) introduced a path integral formula to recover the displacement vector field $\bxi$ corresponding to a given symmetric covariant strain field $\bepsilon$ together with a compatibility condition ensuring that this quadrature does not depend on the path of integration. Almost at the same time, Cesàro \cite{Ces1906} simplified Volterra's approach while leading to the same result.

To calculate the vector $\bxi(\xx_{1})$ at any point $\xx_{1}$ of a deformed medium delimited by a domain $U$ embedded in the Euclidean space $\espace \simeq (\RR^3,\bq)$, these authors fix an initial point $\xx_{0}\in U$ and a smooth path
\begin{equation*}
  t\in[0,1]\longmapsto \xx_{t} \in U,
\end{equation*}
joining $\xx_{0}$ to $\xx_{1}$. The fundamental lemma of integral calculus allows them to express the displacement vector $\bxi(\xx_{1})$ using the following quadrature
\begin{equation*}
  \bxi(\xx_{1}) = \bxi(\xx_{0}) + \int_{0}^{1} \dd{}{t}\bxi(\xx_{t})\,\rd t = \bxi(\xx_{0}) + \int_{0}^{1} \nabla \bxi(\xx_{t})\cdot \xx_{t}'\,\rd t .
\end{equation*}
To mark the dependence of the displacement $\bxi$ on $\xx_{1}$, Cesàro \cite{Ces1906} adds an integration constant, rewriting $\xx'_{t}=(\xx_{t}-\xx_{1})'$. Using the decomposition $\nabla\bxi=\widehat{\bepsilon}+ \widehat\bOmega$, into symmetric and skew-symmetric parts, they get
\begin{equation*}
  \bxi(\xx_{1}) = \bxi(\xx_{0}) + \int_{0}^{1}  \widehat{\bepsilon} (\xx_{t}) \cdot \xx_{t}'\,\rd t + \int_{0}^{1}  \widehat\bOmega (\xx_{t})\cdot (\xx_{t}-\xx_{1})'\,\rd t .
\end{equation*}
Then, an integration by parts in the second integral gives
\begin{equation*}
  \bxi(\xx_{1}) = \bxi(\xx_{0}) + \widehat\bOmega(\xx_{0})\cdot(\xx_{1}-\xx_{0}) + \int_{0}^{1} \widehat{\bepsilon}(\xx_{t}) \cdot \xx_{t}'\,\rd t - \int_{0}^{1} \left(\nabla_{\xx_{t}'}\widehat\bOmega(\xx_{t})\right)\cdot(\xx_{t}-\xx_{1})\,\rd t,
\end{equation*}
which is written in components as
\begin{equation*}
  \xi^{i}(\xx_{1}) = \xi^{i}(\xx_{0}) + {\widehat\Omega^{i}}_{j}(\xx_{0}) ({x_{1}}^{j}-{x_{0}}^{j}) + \int_{0}^{1} \tensor{\widehat{\varepsilon}}{^i_{k}}(\xx_{t}) (x^{k}_{t})'\,\rd t - \int_{0}^{1} \partial_{k}{\widehat\Omega^{i}}_{j}(\xx_{t}) (x_{t}^{j}-x_{1}^{j}) (x^{k}_{t})'\,\rd t.
\end{equation*}
The fundamental identity
\begin{equation}\label{eq:fundamental-identity}
  \partial_{k}{\widehat\Omega^{i}}_{j} = \delta^{ip} \partial_{k}\Omega_{pj} = \frac{1}{2}\delta^{ip}\left(\partial_{k}\partial_{j}\xi_{p} - \partial_{k}\partial_{p}\xi_{j}\right) = \delta^{ip}\left(\partial_{j}\varepsilon_{pk} - \partial_{p}\varepsilon_{jk}\right),
\end{equation}
allows them to eliminate $\partial_{k}\widehat{\Omega}_{ij}$ and to obtain finally
\begin{small}
  \begin{equation*}\label{eq:Cesaro-Voltera}
    \xi^{i}(\xx_{1}) = \xi^{i}(\xx_{0}) + \tensor{\widehat{\Omega}}{^i_j}(\xx_{0}) ({x_{1}}^{j}-{x_{0}}^{j})
    + \int_{0}^{1} \left[ \tensor{\widehat{\varepsilon}}{^i_{k}}(\xx_{t}) + (x_{1}^{j}-x_{t}^{j}) \tensor{\delta}{^{ip}} (\partial_{j}\varepsilon_{pk}(\xx_{t}) - \partial_{p}\varepsilon_{jk}(\xx_{t})) \right] (x^{k}_{t})'\,\rd t.
  \end{equation*}
\end{small}
In more intrinsic notations, and using the identity
\begin{equation*}
  \tensor{\epsilon}{^i_{jl}}\tensor{\epsilon}{^{lp}_q} = \delta^{ip}\delta_{jq} - \tensor{\delta}{^i_q}\tensor{\delta}{_j^p},
\end{equation*}
this Cesàro-Voltera formula recasts as:
\begin{equation*}
  \bxi(\xx) = \bxi(\xx_{0}) + \widehat{\Omega}(\xx_{0})\cdot (\xx-\xx_{0})+\int_{0}^{1} \widehat{\varepsilon}(\xx_{t})\cdot \xx_{t}'\, \rd t - \int_{0}^{1} \left(\xx-\xx_{t}\right)\times\left[(\rot^c\widehat{\bepsilon})(\xx_{t})\cdot \xx_{t}'\right]\rd t,
\end{equation*}
where $\times$ is the vector product on $\RR^3$, $\bxi(0)\in\RR^3$ and $\widetilde{\Omega}(\xx_{0})\in\mathfrak{so}_{3}$ (the space of skew-symmetric linear operators).

It remains to show that the previous path integral does not depend of the path chosen. In order to do that, which will lead to a compatibility condition, they introduce (in modern language) the one-forms
\begin{equation*}
  \balpha^{i} := \left(\tensor{\widehat{\varepsilon}}{^i_{k}}(\xx) + ({x_{1}}^{j}-x^{j}) \delta^{ip} (\partial_{j}\varepsilon_{pk}(\xx) - \partial_{p}\varepsilon_{jk}(\xx)) \right) \rd x^{k}, \qquad i = 1,2,3,
\end{equation*}
so that, if $\Gamma$ denotes the path $\xx_{t}$, we have
\begin{equation*}
  \int_{\Gamma} \balpha^{i} = \int_{0}^{1} \left[\tensor{\widehat{\varepsilon}}{^i_{k}}(\xx_{t}) + (x_{1}^{j}-x_{t}^{j}) \delta^{ip} (\partial_{j}\varepsilon_{pk}(\xx_{t}) - \partial_{p}\varepsilon_{jk}(\xx_{t})) \right](x^{k}_{t})'  \,\rd t.
\end{equation*}
Now, a necessary condition which ensures that this integral depends only on the endpoints and not of the path joining these two points is that the one-forms $\alpha_i$ are \emph{closed}, meaning that $\rd \balpha^{i} = 0$, where
\begin{equation*}
  (\rd \balpha^{i})_{kl} = \partial_{k}{\alpha^{i}}_{l} - \partial_{l}{\alpha^{i}}_{k} = ({x_{1}}^{j}-x^{j})\delta^{ip}\left[\partial_{k}(\partial_{j}\varepsilon_{pl}(\xx) - \partial_{p}\varepsilon_{jl}(\xx)) - \partial_{l}(\partial_{j}\varepsilon_{pk}(\xx) - \partial_{p}\varepsilon_{jk}(\xx))\right].
\end{equation*}
All these quantities vanish identically for every point $\xx_{1}$ if and only if
\begin{equation*}
  \partial_{k}\partial_{j}\varepsilon_{pl} - \partial_{k}\partial_{p}\varepsilon_{jl} - \partial_{l}\partial_{j}\varepsilon_{pk} + \partial_{l}\partial_{p}\varepsilon_{jk} = 0, \qquad \forall j,k,l,p.
\end{equation*}
This condition is designed as the vanishing of the fourth-order tensor field $\bW$, called the \emph{Saint--Venant tensor}, and given in components by
\begin{equation}\label{eq:SV-tensor}
  W_{ijkl}:=\partial_l\partial_{k}\varepsilon_{ij}-\partial_l\partial_{i}\varepsilon_{jk}+\partial_{i}\partial_{j}\varepsilon_{kl}-\partial_{k}\partial_{j}\varepsilon_{il}.
\end{equation}

\section{Two equivalent compatibility conditions}
\label{sec:Ricci}

The incompatibility condition provided by Cesàro and Voltera is given by the vanishing of the fourth-order Saint--Venant tensor field $\bW$. There is however a different compatibility criteria, widely used in the literature (e.g. \cite{SKE2020,Sal2005}), represented by the vanishing of the second-order tensor field
\begin{equation}\label{eq:rot-rot-conditions}
  \Inc \widehat{\bepsilon} := \rot^c\rot^r \widehat{\bepsilon}.
\end{equation}
It is therefore natural to ask what is the link between these two tensors.

In components, the covariant version of $\Inc \widehat{\bepsilon}$ can be written, by \eqref{eq:rot-c} and \eqref{eq:rot-r}, as
\begin{equation*}
  \left(\Inc\bepsilon\right)_{kl} = \tensor{\epsilon}{_{k}^{pi}}\tensor{\epsilon}{_{l}^{rj}}\partial_p\partial_r\varepsilon_{ij}.
\end{equation*}
Now, using the following identity
\begin{equation*}
  \tensor{\epsilon}{_{k}^{pi}}\tensor{\epsilon}{_l^{rj}} =
  \begin{vmatrix}
    \tensor{\delta}{_{kl}} & \tensor{\delta}{_{k}^r} & \tensor{\delta}{_{k}^j} \\
    \tensor{\delta}{^p_l}  & \tensor{\delta}{^{pr}}  & \tensor{\delta}{^{pj}}  \\
    \tensor{\delta}{^i_l}  & \tensor{\delta}{^{ir}}  & \tensor{\delta}{^{ij}}
  \end{vmatrix},
\end{equation*}
and a few calculations, one gets
\begin{equation*}
  \left(\Inc\bepsilon\right)_{kl} = - \tensor{\delta}{^{ij}} W_{ijkl} + \frac{1}{2} (\tensor{\delta}{^{pr}}\tensor{\delta}{^{ij}}W_{ijpr}) \tensor{\delta}{_{kl}}.
\end{equation*}

We shall summarize this result in the following proposition.

\begin{prop}\label{prop:trace}
  The symmetric covariant second-order tensor field  $\Inc \bepsilon$ is entirely determined by the trace of the fourth-order Saint--Venant tensor field $\bW$:
  \begin{equation*}
    \Inc\bepsilon = -\tr_{12}\bW+\Scal\, \bq,
    \quad \text{where} \quad
    \Scal :=  \frac{1}{2}\tr\left[\tr_{12}\bW\right] = \tr\left[\Inc \bepsilon\right],
  \end{equation*}
  or, in components:
  \begin{equation*}
    (\Inc\bepsilon)_{kl} = -\delta^{ij} \bW_{ijkl} + \Scal\,\delta_{kl}
    \quad \mathrm{where} \quad
    \Scal = \frac{1}{2}\delta^{kl}\delta^{ij}\bW_{ijkl}.
  \end{equation*}
\end{prop}

There is a close connection between the Saint-Venant tensor and the Riemann curvature tensor. Namely if $W_{ijkl}$ is a Saint-Venant tensor, then
\begin{equation*}
  R_{ijkl} := W_{ikjl}
\end{equation*}
has all the index symmetries of the Riemann curvature tensor field. It is well-known that in dimension $3$, the Riemann curvature tensor is entirely determined by its trace
\begin{equation*}
  R_{ik} := \delta^{jl}R_{ijkl},
\end{equation*}
called the \emph{Ricci tensor} field. More precisely, in dimension 3, the Riemann curvature tensor can be reconstructed knowing the Ricci tensor. This is false in dimension $d \ge 4$, since another independent traceless fourth-order tensor, known as the \emph{Weyl tensor}, is required in addition to the Ricci tensor to reconstruct the Riemann curvature tensor from its irreducible components (see~\cite{GHL2004}). Due to the close connection between the Saint-Venant tensor and the Riemann curvature tensor, a similar process applies and allows us to reconstruct the Saint-Venant tensor from the incompatibility tensor $\Inc\bepsilon$. This reconstruction is based on an adaptation of the \emph{Kulkarni-Nomizu product} defined in~\cite{GHL2004}. More precisely, given two symmetric covariant second-order tensors $\ba$ and $\bb$, we set
\begin{equation*}
  (\ba\owedge\bb)_{ijkl}=a_{ij}b_{kl}+a_{kl}b_{ij}-a_{jk}b_{il}-a_{il}b_{jk},
\end{equation*}
which is a covariant fourth-order tensor, having the index symmetries of the Saint--Venant tensor, and we have the following result.

\begin{prop}\label{prop:composantes irreductibles}
  In dimension 3, the covariant fourth-order Saint--Venant tensor $\bW$ is entirely determined by the symmetric covariant second-order incompatibility tensor $\Inc \bepsilon$:
  \begin{equation*}
    \bW = -\Inc \bepsilon\owedge\bq+\frac{1}{2}\Scal\bq\owedge\bq,
    \qquad\text{where}\qquad
    \Scal = \tr\left[\Inc \bepsilon\right] = \frac{1}{2}\tr\left[\tr_{12}\bW\right].
  \end{equation*}
  In components, we get
  \begin{equation*}
    \bW_{ijkl} = -(\Inc\bepsilon)_{ij}\delta_{kl} - (\Inc\bepsilon)_{kl}\delta_{ij} + (\Inc\bepsilon)_{jk}\delta_{il} + (\Inc\bepsilon)_{il}\delta_{jk} + \frac{\Scal}{2}(\delta_{ij}\delta_{kl}-\delta_{il}\delta_{jk}).
  \end{equation*}
\end{prop}

\begin{rem}
  In particular, in dimension 3, the conditions $\bW=0$ and $\Inc \bepsilon=0$ are equivalent. In higher dimension, this result is false. We can have $\Inc \bepsilon=0$, whereas $\bW \ne 0$. In dimension 2, the equation $\bW=0$ reduces to $\Scal=0$.
\end{rem}

\begin{proof}
  Consider the linear mapping
  \begin{equation*}
    \ba \longmapsto  \ba\owedge\bq, \qquad S^{2}\RR^{3} \to S^{2,2}\RR^{3},
  \end{equation*}
  where $S^{2}\RR^{3}$ is the space of symmetric second-order tensors and $S^{2,2}\RR^{3}$, the space of fourth-order tensors having the symmetry of the Saint-Venant tensor. If $\ba\owedge\bq = 0$, then, $\tr_{12}(\ba\owedge\bq) = \ba + (\tr \ba)\bq = 0$ and thus $\ba=0$, this mapping is injective. But $\dim S^{2}\RR^{3} = \dim S^{2,2}\RR^{3} = 6$, and this mapping is thus a linear isomorphism.

  Therefore, given $\bW \in S^{2,2}\RR^{3}$, one can find a unique $\ba \in S^{2}\RR^{3}$ such that $\bW = \ba\owedge\bq$. We have then successively
  \begin{equation*}
    \tr_{12} \bW = \ba + (\tr \ba)\bq, \quad \text{and} \quad \tr (\tr_{12} \bW) = 4 \tr \ba.
  \end{equation*}

  Now, by proposition~\ref{prop:trace}, we have
  \begin{equation*}
    \Inc\bepsilon = - \tr_{12}\bW + \Scal\, \bq,
    \quad \text{where} \quad
    \Scal :=  \frac{1}{2}\tr\left[\tr_{12}\bW\right] = \tr\left[\Inc \bepsilon\right],
  \end{equation*}
  from which we deduce that
  \begin{equation*}
    \ba = - \Inc\bepsilon + \frac{1}{2} \Scal\, \bq.
  \end{equation*}
  We have therefore
  \begin{equation*}
    \bW = -\Inc \bepsilon\owedge\bq+\frac{1}{2}\Scal\bq\owedge\bq .
  \end{equation*}
\end{proof}

\begin{rem}
  The fact that the Saint--Venant tensor reduces to a function in 2D and to a symmetric second-order tensor in 3D can also be understood using the Hodge star operator introduced at the end of \autoref{sec:deRham-complex}, see also \cite{Kroe1958,GP2004}.
\end{rem}

\section{The de Rham complex}
\label{sec:deRham-complex}

The exterior derivative extends the differential of a function to \emph{differential forms}, meaning \emph{alternate} covariant tensor fields (which change sign under the transposition of any pair of indices). More precisely, denoting by $\bOmega^{k}(M)$ the space of differential $k$-forms on a manifold $M$, the exterior derivative $\mathrm{d}$ is a differential operator of order one which transforms a differential $k$-form $\balpha\in\bOmega^{k}(M)$ into a differential $(k+1)$-form $\rd \balpha\in\bOmega^{k+1}(M)$, satisfying the rule called \emph{Cartan magic formula}
\begin{equation}\label{eq:Cartan-formula}
  \Lie_{X}\balpha = \rd  \iota_{X}\balpha + \iota_{X} \rd \balpha,
\end{equation}
for all vector fields $X$, where $\Lie_{X}$ is the Lie derivative and $\iota_{X}$ means the contraction of $X$ with $\alpha$ (if $\balpha \in \bOmega^{k}(M)$, then, $\iota_{X}\balpha = X \cdot \balpha \in \bOmega^{k-1}(M)$ and if $f \in \bOmega^{0}(M)$, then, $\iota_{X} f = 0$). This formula provides a recursive way to calculate the exterior derivative $\rd$. In particular, given a local coordinate system $(x^{i})$ on $M$, we get, in components:
\begin{align*}
  (\rd_{0} f)_{i}        & = \partial_{i}f, \qquad \text{for a function $f \in \bOmega^{0}(M)$},
  \\
  (\rd_{1}\balpha)_{ij}  & = \partial_{i}\alpha_{j} - \partial_{j}\alpha_{i}, \qquad \text{for a one-form $\balpha \in \bOmega^{1}(M)$},
  \\
  (\rd_{2}\bbbeta)_{ijk} & = \partial_{i}\beta_{jk} + \partial_{j}\beta_{ki} + \partial_{k}\beta_{ij}, \qquad \text{for a two-form $\bbeta \in \bOmega^{2}(M)$}.
\end{align*}

The exterior derivative generates a graded sequence of differential forms spaces called the \emph{de Rham differential complex}
\begin{equation*}
  \bOmega^{0}(M) \overset{\rd_{0}}{\longrightarrow} \bOmega^{1}(M) \overset{\rd_{1}}{\longrightarrow} \bOmega^2(M) \overset{\rd_{2}}{\longrightarrow} \bOmega^3(M) \overset{\rd_{3}}{\longrightarrow} \dotsb,
\end{equation*}
with the fundamental property that $\rd^{2} = \rd  \circ \rd  = 0$.

\begin{rem}
  In particular, when $U$ is an open subset of the Euclidean space $\RR^3$, then, the de Rham complex corresponds to the sequence
  \begin{equation*}
    \bOmega^{0}(U) \overset{\rd_{0}=\grad}{\scalebox{2}[1]{$\longrightarrow$}} \bOmega^{1}(U) \overset{\rd_{1}=\rot}{\scalebox{2}[1]{$\longrightarrow$}} \bOmega^2(U) \overset{\rd_{2}=\dive}{\scalebox{2}[1]{$\longrightarrow$}} \bOmega^3(U),
  \end{equation*}
  and the well-known properties $\rot\circ\grad=0$ and $\dive\circ\rot=0$ are special cases of the rule $\rd  \circ \rd  =0$.
\end{rem}

A $k$-form $\balpha$ on a manifold $M$ is \emph{closed} if $\rd_{k} \balpha = 0$. It is \emph{exact} if there exists a $(k-1)$-form $\bbbeta$ such that $\balpha = \rd_{k-1} \bbbeta$. Since $\rd  \circ \rd =0$, an exact form is always closed but the converse is not always true. However, when $M=U$ is an \emph{open convex set} of $\RR^{n}$, then, the converse is true. An easy way to prove this result is to build an integral operator
\begin{equation}\label{def:Poincare}
  \mathcal{K}_{k} \colon \bOmega^{k+1}(U) \longmapsto \bOmega^{k}(U),
\end{equation}
such that
\begin{equation}\label{eq:Rham-homotopy}
  \mathcal{K}_{k} \rd_{k} \balpha + \rd_{k-1} \mathcal{K}_{k-1}\balpha = \balpha, \qquad \forall \balpha \in \bOmega^{k}(U),
\end{equation}
and called an \emph{homotopy formula}. Indeed, if $\rd_{k}\balpha=0$, then, we have $\balpha = \rd_{k-1} \bbbeta$ where $\bbbeta = \mathcal{K}_{k-1}\balpha$. This is why $\mathcal{K}$ is called an \emph{integrator}. Moreover, such an integrator can be constructed explicitly. Without loss of generality, we can assume that $U$ contains the origin $0$. We shall then introduce the \emph{Poincaré integrator}
\begin{equation}\label{eq:Poincare-integrator}
  \mathcal{K}_{k-1} \balpha := \int^{0} _{-\infty} (\phi^{t})^{*}\iota_{X}\balpha \, \rd t,
\end{equation}
for $\alpha \in \bOmega^{k}(U)$, where $X(\xx) = \xx$ is the radial vector field, $\phi^{t}(\xx) = e^{t} \xx$, its flow and $(\phi^{t})^{*}$ means the pullback by $\phi^{t}$. This formula may be recast in a possibly more readable form as
\begin{equation*}
  (\mathcal{K}_{k-1} \balpha)(\xx) = \int^{1} _{0} s^{k-1} \xx \cdot \balpha(s\xx) \, \rd s .
\end{equation*}

When the $n$-dimensional manifold $M$ is equipped with a Riemannian metric $g$ (or more generally only a volume form), the De Rham complex
\begin{equation*}
  \bOmega^{0}(M) \overset{\rd  _0}{\longrightarrow} \bOmega^{1}(M) \overset{\rd  _1}{\longrightarrow} \bOmega^2(M) \overset{\rd  _2}{\longrightarrow} \dotsb \overset{\rd  _{n-1}}{\longrightarrow} \bOmega^{n}(M)
\end{equation*}
induces a \emph{dual complex}
\begin{equation*}
  \bOmega^0(M)^{*}\overset{\rd^{(1)*} _0}{\longleftarrow}
  \bOmega^1(M)^{*}\overset{\rd^{(1)*} _1}{\longleftarrow}
  \bOmega^2(M)^{*}\overset{\rd^{(1)*} _2}{\longleftarrow}
  \dotsb \overset{\rd^{(1)*} _{n-1}}{\longleftarrow}
  \bOmega^{n}(M)^{*},
\end{equation*}
where $\bOmega^{k}(M)^{*}$ is the space of alternate contravariant tensor fields of order $k$ and $\rd^{*}$ is the \emph{formal adjoint} of the exterior derivative $\rd$ (also called the \emph{co-differential}). It is defined implicitly by
\begin{equation*}
  \left\langle\alpha,\rd^{*}\bP\right\rangle := \int_{M} ( \alpha , \rd^{*} \bP )\vol_{g} =  \int_{M} ( \rd  \alpha , \bP ) \vol_{g} = \left\langle\rd\alpha,\bP\right\rangle ,
\end{equation*}
for all $\alpha \in \bOmega^{k-1}(M)$ and $\bP \in \bOmega^{k}(M)^{*}$ with compact support, where $\vol_{g}$ is the Riemannian volume element and $(\cdot, \cdot)$ means the total contraction between a covariant alternate tensor and a contravariant alternate tensor, both of order $k$. It inherits the property $\rd^{*} \circ \rd^{*}= 0$, from $\rd \circ \rd= 0$.

Introducing the \emph{Hodge star operator}
\begin{equation*}
  \star \colon \Omega^{k}(\RR^{n}) \to  \Omega^{n-k}(\RR^{n})^{*}, \qquad \omega_{i_{1} \dotsc i_{k}} \mapsto (\star \omega)^{j_{1} \dotsc j_{n-k}} = \epsilon^{i_{1} \dotsc i_{k}j_{1} \dotsc j_{n-k}}\omega_{i_{1} \dotsc i_{k}},
\end{equation*}
the co-differential can be rewritten (see \cite{Lan1999,GHL2004}) as
\begin{equation*}
  \rd^{*} = \star^{-1} \rd \star.
\end{equation*}

A remarkable application of the de Rham complex in Physics is the intrinsic geometric formulation of Maxwell equations in the Minkowski spacetime $(\RR^{4},\bbeta)$ where
\begin{equation*}
  \bbeta = - c^{2} {\rd t}^{2} + {\rd x}^{2} + {\rd y}^{2} + {\rd z}^{2}
\end{equation*}
is the Lorentz metric. The electric field $\bE = (E^i)$ and the magnetic induction $\bB = (B^i)$ have no intrinsic geometric meaning (they depend in fact on the observer). It is better to represent them, in a unified way, as the components of a $2$-form $\mathcal{F}$ on $\RR^{4}$, called the \emph{Faraday tensor}, which is written in components as
\begin{equation*}
  \mathcal{F} = (\mathcal{F}_{\mu \nu})=
  \begin{bmatrix}
    0    & E^1    & E^2    & E^3    \\
    -E^1 & 0      & -B^{3} & B^{2}  \\
    -E^2 & B^{3}  & 0      & -B^{1} \\
    -E^3 & -B^{2} & B^{1}  & 0
  \end{bmatrix}.
\end{equation*}
Using this formalism, the first two Maxwell equations (Maxwell-Faraday equation and Maxwell-Thomson equation)
\begin{equation*}
  \rot \bE + \frac{\partial\bB}{\partial t} = 0,
  \qquad
  \dive \bB = 0,
\end{equation*}
recast simply as the unique equation
\begin{equation}\label{eq:Maxwell-1}
  \rd \mathcal{F} = 0,
\end{equation}
which states that $\mathcal{F}$ is closed. Poincaré's lemma ensures then, at least locally, that there exists a $1$-form $\mathcal{A}$ such that $\mathcal{F} = d\mathcal{A}$, which can be moreover calculated using the Poincaré integrator~\eqref{eq:Poincare-integrator}. Designing its components by $\mathcal{A}:=(U,-A_{1},-A_{2},-A_{3})$, we recover the 3D magnetic potential vector $A:=(A_{1},A_{2},A_{3})$ and the electric potential $U$. The equation $\mathcal{F} = d\mathcal{A}$ means that
\begin{equation*}
  \bE = -\grad U - \partial_t A, \qquad \bB = \rot A.
\end{equation*}

To complete the system of Maxwell's equations, one must first introduce two new variables, the magnetic field $\bH = (H^{i})$ and the electric displacement $\bD = (D^{i})$, which are unified by introducing the $2$-forms $\mathcal{D}$, called the \emph{electromagnetic displacement tensor} and represented in components by
\begin{equation*}
  \mathcal{D} = (\mathcal{D}^{\mu \nu})=
  \begin{bmatrix}
    0   & -D^1   & -D^2   & -D^3   \\
    D^1 & 0      & -H^{3} & H^{2}  \\
    D^2 & H^{3}  & 0      & -H^{1} \\
    D^3 & -H^{2} & H^{1}  & 0
  \end{bmatrix}.
\end{equation*}
The last two Maxwell equations (Maxwell-Ampère equation and Maxwell-Gauss equation) are written as
\begin{equation*}
  \rot\bH-\partial_t\bD = \bJ, \qquad \dive\bD = \rho,
\end{equation*}
where $\rho$ is the charge density and $\bJ$ is the current density, which can be gathered into the quadrivector
\begin{equation*}
  \mathcal{J} := (\rho, J^{1}, J^{2}, J^{3}).
\end{equation*}
Using the intrinsic formalism, the last two Maxwell equations can be unified into the unique equation
\begin{equation}\label{eq:Maxwell-2}
  \rd^{*} \mathcal{D} = \mathcal{J},
\end{equation}
where $d^{*}$ is the \emph{co-differential} (or formal adjoint of $d$). Since $\rd^{*} \circ \rd^{*}= 0$, we get $\rd^{*} \mathcal{J} = 0$, which corresponds to the \emph{continuity equation}
\begin{equation*}
  \partial_t\rho + \dive \bJ = 0.
\end{equation*}
However, if the quadrivector current $\mathcal{J}$ is given, the full system of Maxwell equations \eqref{eq:Maxwell-1} and \eqref{eq:Maxwell-2} is under-determinate since they correspond to 8 scalar equations for 12 scalar variables represented by the components of the variables $\mathcal{F}$ and $\mathcal{D}$. Therefore, in order to close the system, one needs to introduce a \emph{constitutive law} between $\mathcal{F}$ and $\mathcal{D}$. If one remains in the linear domain, such a law is written as
\begin{equation*}
  \mathcal{D}^{\mu\nu} = \tensor{C}{^{\mu\nu\rho\kappa}}\mathcal{F}_{\rho\kappa}.
\end{equation*}

\begin{exam}
  For instance, for Maxwell's equations in the vacuum, the constitutive law is isotropic and written as
  \begin{equation*}
    \mathcal{D}^{\mu\nu} = \frac{1}{\mu_{0}} \eta^{\mu\rho}\eta^{\nu\kappa}\mathcal{F}_{\rho\kappa},
  \end{equation*}
  where $\mu_{0}$ is the magnetic permeability.
\end{exam}

In order to make later the analogy with the Elasticity complex more explicit, we have summarized this construction by the following diagram, where we have emphasised the connection between the complex and its dual through a constitutive relation between $\mathcal{F}$ and $\mathcal{D}$:
\begin{equation}\label{eq:Tonti-diagram-1}
  \begin{array}{ccccccc}
    \bullet & \stackrel\rd{\longrightarrow}      & \mathcal{A} & \stackrel\rd{\longrightarrow}      & \mathcal{F}               & \stackrel\rd{\longrightarrow}      & 0       \\
            &                                    &             &                                    & \rotatebox{90}{$\bowtie$} &                                    &
    \\
    0       & \stackrel{\rd^{*}}{\longleftarrow} & \mathcal{J} & \stackrel{\rd^{*}}{\longleftarrow} & \mathcal{D}               & \stackrel{\rd^{*}}{\longleftarrow} & \bullet
  \end{array}
\end{equation}

\section{The Elasticity complex}
\label{subsec:elasticity-complex}

The Dubois-Violette theory~\cite{DH1999,DH2002} generalizes de Rham's complexes for tensors with other index symmetries. The idea is to adapt the de Rham complex directly for covariant tensor fields which have the symmetries of the elasticity tensors rather than alternate forms. We insist on the fact that this construction is much more natural and simple than the BGG formalism \cite{Eas1999,Eas2000,SKE2020,AFW2006, Fal2008}.

This construction relies on Young's symmetrizations (\autoref{sec:Dubois-Violette}) which we illustrate first by reformulating de Rham's complex. Rather than defining the exterior derivative using Cartan magic formula \eqref{eq:Cartan-formula}, we could have used the canonical covariant derivative $\nabla$ on $\RR^{n}$ and define it on $k$-forms by the following formula
\begin{equation*}
  (\rd_{k}\alpha) (\bX_{1}, \dotsc, \bX_{k+1}) = \sum_{\sigma \in \mathcal{S}_{k+1}} \varepsilon(\sigma) \left(\nabla_{\bX_{\sigma(k+1)}} \alpha \right)(\bX_{\sigma(1)}, \dotsc, \bX_{\sigma(k)})
\end{equation*}

To build the Elasticity complex in a similar way, we need only to use the index symmetries of the involved tensors. This is done using the generalized Dubois--Violette complex (\autoref{sec:Dubois-Violette}) with $N = 2$. In dimension 3, it is given by
\begin{equation}\label{eq:spin2-complex}
  \Omega_{2}^{0}(\RR^{3})\ \stackrel{\rd_{0}}{\to } \  \Omega_{2}^{1}(\RR^{3})\ \stackrel{\rd_{1}}{\to } \ \Omega_{2}^{2}(\RR^{3}) \ \stackrel{\rd_{2}}{\to } \ \Omega_{2}^{3}(\RR^{3}) \ \stackrel{\rd_{3}}{\to } \ \Omega_{2}^{4}(\RR^{3}) \ \stackrel{\rd_{4}}{\to } \ \Omega_{2}^{5}(\RR^{3}) \ \stackrel{\rd_{5}}{\to } \ \Omega_{2}^{6}(\RR^{3}).
\end{equation}
The symmetries of the tensor spaces $\Omega_{2}^{k}(\RR^{3})$ are encoded by the following sequence of Young's tableaux
\begin{equation*}
  \bullet \qquad \young(1) \qquad \young(12) \qquad \young(12,3) \qquad \young(12,34) \qquad \young(12,34,5) \qquad \young(12,34,56)
\end{equation*}
and where the differential operators $\rd_{k}$ are given by (\ref{exam:DV-derivatives}).
One significative difference with the de Rham complex is that the sequence of operators $\rd_{k}$ satisfies
\begin{equation*}
  \rd_{k+2} \, \rd_{k+1} \, \rd_{k} = 0,
\end{equation*}
or, in short terms $\rd^{3}= 0$, rather than $\rd^{2} = 0$ for the de Rham complex. To recover such a formula, it is necessary to "skip steps", giving rise to the following Elasticity complex
\begin{equation}\label{def:elasticity-complex}
  \Omega_{2}^{1}(\RR^{3}) \ \stackrel{D_{1}}{\longrightarrow} \
  \Omega_{2}^{2}(\RR^{3}) \ \stackrel{D_{2}}{\longrightarrow} \
  \Omega_{2}^{4}(\RR^{3}) \ \stackrel{D_{3}}{\longrightarrow} \ \Omega_{2}^{5}(\RR^{3}),
\end{equation}
where
\begin{align*}
   & (D_{1}\bxi^\flat)_{ij} := (\rd_{1}\bxi^\flat)_{ij}  = \frac{1}{2}(\partial_{j}\xi_{i}+\partial_{i}\xi_{j}),
  \\
   & ({D_{2}}\bepsilon)_{ijkl} := (\rd_{3} \, \rd_{2}\bepsilon)_{ijkl} = \frac{1}{3}\left(\partial_l\partial_{k}\varepsilon_{ij} - \partial_l\partial_{i}\varepsilon_{jk}+\partial_{i}\partial_{j}\varepsilon_{kl} - \partial_{k}\partial_{j}\varepsilon_{il}\right),
  \\
   & ({D_3}\bW)_{ijklm} :=  (\rd_{4}\bW)_{ijklm} = \frac{1}{2} (\partial_{i}W_{kjml} + \partial_{k}W_{mjil} + \partial_{m}W_{ijkl}).
\end{align*}
and which satisfies $D_{j+1} \circ D_{j} = 0$. We deduce thus immediately the following result.
\begin{prop}[Saint--Venant compatibility condition]
  If the symmetric covariant second-order tensor field of infinitesimal strain $\bepsilon$ derives from a displacement covector field $\bxi^\flat$, that is $\bepsilon = D_{1}\bxi^\flat$, then,
  \begin{equation*}
    D_{2}\bepsilon = D_{2}D_{1}\bxi^\flat = 0,
    \qquad\text{i.e.}\qquad
    \partial_l\partial_{k}\varepsilon_{ij}-\partial_l\partial_{i}\varepsilon_{jk}+\partial_{i}\partial_{j}\varepsilon_{kl}-\partial_{k}\partial_{j}\varepsilon_{il}=0.
  \end{equation*}
\end{prop}

The remarkable result is that the Elasticity complex is locally exact like the de Rham complex and we can produce an explicit formula for the displacement. This is due to the existence of an \emph{homotopy formula}
\begin{equation*}
  D_{1}K_{1}\bepsilon + K_{2}D_{2}\bepsilon = \bepsilon, \qquad \forall \bepsilon \in \Omega_{2}^{2}(\RR^{3}),
\end{equation*}
where $ K_{1} \colon \Omega_{2}^{2}(\RR^{3}) \to  \Omega_{2}^{1}(\RR^{3})$ and $K_{2} \colon \Omega_{2}^{4}(\RR^{3}) \to  \Omega_{2}^{2}(\RR^{3})$ are integral operators. More precisely, we introduce first the integrator
\begin{equation}\label{eq:obstruction}
  K_{2} \bW(\xx)= \int_{0}^{1} \rd t  \int_{0}^{t} s \bW(s \xx):(\xx \otimes \xx) \,ds .
\end{equation}
whose components are
\begin{equation*}
  (K_{2} \bW)_{ij} = \int_{0}^{1} \rd t  \int_{0}^{t} sx^{k}x^{l} \bW_{ijkl}(sx)\, ds .
\end{equation*}
Using the relations
\begin{equation*}
  \frac{d}{\rd t}\left(t\varepsilon_{ij}(tx)\right) = \varepsilon_{ij}(tx) + tx^{k}\pdi{k}\varepsilon_{ij}(tx),
  \qquad
  \frac{d}{ds}\left(s\pdi{k}\varepsilon_{ij}(sx)\right) = \pdi{k}\varepsilon_{ij}(sx) + sx^{l}\pdij{kl}\varepsilon_{ij}(sx),
\end{equation*}
and integrating by parts, we deduce that
\begin{align*}
  (K_{2} D_{2}\varepsilon)_{ij}(x) & = \varepsilon_{ij}(x) - \int_{0}^{1} \varepsilon_{ij}(tx) \, \rd t - \int_{0}^{1} tx^{k}\left( \pdi{i}\varepsilon_{jk}(tx) +  \pdi{j}\varepsilon_{ik}(tx)\right) \rd t
  \\
                                   & \quad + \int_{0}^{1} \rd t  \int_{0}^{t} x^{k}\left( \pdi{i}\varepsilon_{jk}(sx) + \pdi{j}\varepsilon_{ik}(sx) -\partial_{k}\varepsilon_{ij}(sx)\right) ds
  \\
                                   & \quad + \int_{0}^{1} \rd t  \int_{0}^{t} s x^{k}x^{l} \pdij{ij}\varepsilon_{kl}(sx)\, ds.
\end{align*}
Now, we look for an integral operator $K_{1} \colon \Omega_{2}^{2}(\RR^{3}) \to  \Omega_{2}^{1}(\RR^{3})$, depending on the first jet of $\bepsilon$, and such that
\begin{equation*}
  K_{2}D_{2}\varepsilon + D_{1}K_{1}\varepsilon = \varepsilon, \qquad \forall \varepsilon.
\end{equation*}
Therefore, we are looking for $K_{1}\bepsilon$ as a linear combination of
\begin{align*}
  (A\bepsilon)_{i} & := \int_{0}^{1} x^{k} \varepsilon_{ik}(tx) \,\rd t ,
  \\
  (B\bepsilon)_{i} & := \int_{0}^{1} \rd t  \int_{0}^{t} x^{k}x^{l} \pdi{l}\varepsilon_{ik}(sx) \,ds ,
  \\
  (C\bepsilon)_{i} & := \int_{0}^{1} \rd t  \int_{0}^{t} x^{k}x^{l} \pdi{i}\varepsilon_{kl}(sx) \,ds.
\end{align*}
Calculating
\begin{align*}
  (D_{1}A\varepsilon)_{ij} & := \int_{0}^{1} \varepsilon_{ij}(tx) \, \rd t + \frac{1}{2} \int_{0}^{1} tx^{k}\left( \pdi{i}\varepsilon_{jk}(tx) + \pdi{j}\varepsilon_{ik}(tx)\right) \rd t,
  \\
  (D_{1}B\varepsilon)_{ij} & :=\int _{0} ^{1} \rd t\int_{0} ^{t} x^{k}\partial_{k}\varepsilon_{ij}(sx)\,ds+
  \frac{1}{2} \int_{0}^{1} t  x^{k}\left( \pdi{i}\varepsilon_{jk}(sx) + \pdi{j}\varepsilon_{ik}(sx) \right) \rd t
  \\
  (D_{1}C\varepsilon)_{ij} & := \int_{0}^{1} \rd t  \int_{0}^{t} x^{k}\left( \pdi{i}\varepsilon_{jk}(sx) + \pdi{j}\varepsilon_{ik}(sx) \right) ds +  \int_{0}^{1} \rd t  \int_{0}^{t} s x^{k}x^{l} \pdij{ij}\varepsilon_{kl}(sx)\, ds,
\end{align*}
we observe that
\begin{equation*}
  (K_{2} D_{2}\varepsilon)_{ij} = \varepsilon_{ij} - (D_{1}A\varepsilon)_{ij} -  (D_{1}B\varepsilon)_{ij} + (D_{1}C\varepsilon)_{ij}.
\end{equation*}
By identification with the homotopy formula we deduce that
\begin{equation*}
  K_{1} =  A +  B - C
\end{equation*}
is the solution to the problem. A calculation and a simple integration by parts gives
\begin{equation*}
  (K_{1}\varepsilon)_{i}
  = \int_{0}^{1} x^{k} \varepsilon_{ik}(tx) \,\rd t
  + \int_{0}^{1} (1-t)\,x^{k}x^{l}\left( \pdi{l}\varepsilon_{ik}(tx)-  \pdi{i}\varepsilon_{kl}(tx)\right)\rd t,
\end{equation*}
which is exactly the Cesàro-Volterra formula \eqref{eq:Cesaro-Voltera}.

\begin{rem}
  When the order two symmetric tensor $K_{2}\bW\neq 0$ it represents an obstruction term measuring the degree of incompatibility of the deformation.
\end{rem}

\section{The dual Elasticity complex and stress potentials}
\label{sec:stress-potentials}

The Airy potential~\cite{TGA1970,Mus1977,Sad2009} was introduced by George Biddell Airy~\cite{Air1863} in the 19th century, as a scalar function to simplify the resolution of plane stress problems in 2D. In the two-dimensional case, it automatically satisfies the equilibrium equations by expressing the stresses as second derivatives of the Airy function:
\begin{equation*}
  \sigma_{xx} = \frac{\partial^2 \varphi}{\partial y^2}, \quad \sigma_{yy} = \frac{\partial^2 \varphi}{\partial x^2}, \quad \sigma_{xy} = -\frac{\partial^2 \varphi}{\partial x \partial y}.
\end{equation*}

Several approaches have been proposed to generalize the Airy potential to three dimensions~\cite{Lov1944,Sch1953,LS1954,Gur1963,Teo1972,Mus1977,Tin1996,Pom2016}. For instance, Maxwell and Morera Stress Functions correspond to different triples of stress functions which allow the representation of a 3D stress field while satisfying the equilibrium equations. They partially generalize the Airy concept but with increased complexity. Beltrami-Schaefer Potentials~\cite{Sch1953} is another generalization which uses combinations of harmonic and biharmonic functions to describe stresses in 3D by introducing the potential as a second-order symmetric tensor. A difficulty with this formulation is, however, that it merges compatibility conditions on the strain and ones on the stress by introducing, \textit{a priori}, the constitutive law into the formulation of the problem. Pommaret in~\cite{Pom2016} seems the first to have clarified the subject using the framework of \emph{Spencer cohomology}~\cite{Spe1969}.

In the present work, we shall use the simpler approach of Dubois--Violette \& Henneaux to recover the 2-dimensional and 3-dimensional stress potentials. It relies on the fact that, like the \emph{co-differential} is usually defined using the \emph{Hodge star operator} \cite{Lan1999} on the de Rham complex, the same procedure applies to generalized complexes. Therefore, we shall denote by $\Omega^{k}_{2}(\RR^{n})^{*}$ the space of \emph{contravariant tensor fields} having the symmetries associated with the Young tableau $(D^{k}_{2})$ (see \autoref{sec:Dubois-Violette}) and we write this dual complex as
\begin{equation}\label{eq:spin2-dual-complex}
  \Omega^{0}_{2}(\RR^{n})^{*} \
  \underset{}{\stackrel{\rd_{0}^{*}}{\longleftarrow}}
  \cdots
  \underset{}{\stackrel{\rd_{k-1}^{*}}{\longleftarrow}}\
  \Omega^{k}_{2}(\RR^{n})^{*} \
  \underset{}{\stackrel{\rd_{k}^{*}}{\longleftarrow}} \
  \cdots
  \
  \underset{}{\stackrel{\rd_{2n-1}^{*}}{\longleftarrow}} \
  \Omega^{2n}_{2}(\RR^{n})^{*}
  ,
\end{equation}
where $\rd_{k}^{*}$ denotes the \emph{generalized co-differential} defined in arbitrary dimension in \autoref{sec:Dubois-Violette}, using the generalized Hodge star operator \eqref{def:generalised-Hodge}. In the following, we are interested in dimension 2 and then, in dimension 3.

In dimension 2, we get
\begin{equation*}\label{eq:2D-complex}
  \Omega_{2}^{0}(\RR^{2})\
  \stackrel{\rd_0}{\longrightarrow} \
  \Omega_{2}^{1}(\RR^{2})\
  \stackrel{\rd_1}{\longrightarrow} \
  \Omega_{2}^{2}(\RR^{2}) \
  \stackrel{\rd_2}{\longrightarrow} \
  \Omega_{2}^{3}(\RR^{2})
  \stackrel{\rd_3}{\longrightarrow} \
  \Omega_{2}^{4}(\RR^{2}),
\end{equation*}
and its dual complex is
\begin{equation*}\label{eq:dual-2D-complex}
  \Omega_{2}^{0}(\RR^{2})^*\
  \stackrel{\rd^* _0}{\longleftarrow} \
  \Omega_{2}^{1}(\RR^{2})^*\
  \stackrel{\rd^* _1}{\longleftarrow} \
  \Omega_{2}^{2}(\RR^{2})^* \
  \stackrel{\rd^* _2}{\longleftarrow} \
  \Omega_{2}^{3}(\RR^{2})^*
  \stackrel{\rd^* _3}{\longleftarrow} \
  \Omega_{2}^{4}(\RR^{2})^*.
\end{equation*}
The Hodge star isomorphisms defined in \eqref{def:generalised-Hodge} are given by
\begin{align*}
  \star_{0} \colon \Omega_{2}^{0}(\RR^{2}) & \to \Omega^{4}_{2}(\RR^{2})^{*}, &  & \varphi          \mapsto A^{ijkl} = \epsilon^{ik}\epsilon^{jl} \varphi,
  \\
  \star_{1} \colon \Omega_{2}^{1}(\RR^{2}) & \to \Omega^{3}_{2}(\RR^{2})^{*}, &  & \xi_{l}          \mapsto \Gamma^{ijk} = \epsilon^{ik}\epsilon^{jl} \xi_{l},
  \\
  \star_{2} \colon \Omega_{2}^{2}(\RR^{2}) & \to \Omega^{2}_{2}(\RR^{2})^{*}, &  & \varepsilon_{kl} \mapsto \sigma^{ij} = \epsilon^{ik}\epsilon^{jl} \varepsilon_{kl},
  \\
  \star_{3} \colon \Omega_{2}^{3}(\RR^{2}) & \to \Omega^{1}_{2}(\RR^{2})^{*}, &  & K_{jkl}          \mapsto \xi^{i} = \epsilon^{ik}\epsilon^{jl} K_{jkl},
  \\
  \star_{4} \colon \Omega_{2}^{4}(\RR^{2}) & \to \Omega^{0}_{2}(\RR^{2})^{*}, &  & W_{ijkl}         \mapsto f = \epsilon^{ik}\epsilon^{jl} W_{ijkl},
\end{align*}
and their inverses by
\begin{align*}
  \star_{0}^{-1} \colon \Omega^{4}_{2}(\RR^{2})^{*} & \to \Omega_{2}^{0}(\RR^{2}), &  & A^{ijkl}     \mapsto \varphi = \frac{1}{4} \epsilon_{ik}\epsilon_{jl} A^{ijkl},
  \\
  \star_{1}^{-1} \colon \Omega^{3}_{2}(\RR^{2})^{*} & \to \Omega_{2}^{1}(\RR^{2}), &  & \Gamma^{ijk} \mapsto \xi_{l} = \frac{1}{2} \epsilon_{ik}\epsilon_{jl} \Gamma^{ijk},
  \\
  \star_{2}^{-1} \colon \Omega^{2}_{2}(\RR^{2})^{*} & \to \Omega_{2}^{2}(\RR^{2}), &  & \sigma^{ij}   \mapsto \varepsilon_{kl} = \epsilon_{ik} \epsilon_{jl} \sigma^{ij},
  \\
  \star_{3}^{-1} \colon \Omega^{1}_{2}(\RR^{2})^{*} & \to \Omega_{2}^{3}(\RR^{2}), &  & \xi^{i}       \mapsto  K_{jkl} = \frac{1}{2} \epsilon_{ik}\epsilon_{jl} \xi^{i},
  \\
  \star_{4}^{-1} \colon \Omega^{0}_{2}(\RR^{2})^{*} & \to \Omega_{2}^{4}(\RR^{2}), &  & f             \mapsto W_{ijkl} = \frac{1}{4} \epsilon_{ik}\epsilon_{jl} f.
\end{align*}
Now, since
\begin{align*}
   & (\rd_{0} \varphi)_{i}       = \partial_{i} \varphi,
  \\
   & (\rd_{1} \bxi^{\flat})_{ij} = \frac{1}{2}(\partial_{j}\xi_{i} + \partial_{i}\xi_{j}),
  \\
   & (\rd_{2}\bepsilon)_{ijk} = \frac{2}{3}(\partial_{k}\varepsilon_{ij} - \partial_{i}\varepsilon_{jk}),
\end{align*}
we deduce therefore that
\begin{align*}
   & (D_{1}^{*} \bsigma)^{i} = \left(\star_{3} \, \rd_{2} \, \star_{2}^{-1} \bsigma\right)^{i} = \frac{4}{3} \partial_{j} \sigma^{ij}
  \\
   & (D_{2}^{*} \bA)^{ij}    = \left(\star_{2} \, \rd_{1} \, \rd_{0} \ \star^{-1}_{0} \bA\right)^{ij} = \left(\star_{2} \, \rd_{1} \, \rd_{0} \ \varphi\right)^{ij} = \epsilon^{ik}\epsilon^{jl} \partial_{k} \partial_{l} \varphi,
\end{align*}
where $\varphi := \star^{-1}_{0} \bA$, and where we have used the following identities
\begin{equation*}
  \epsilon^{ij}\epsilon_{ik} =\delta^j _k, \qquad \epsilon^{ij}\epsilon_{ij} =2 .
\end{equation*}
Therefore, the exactness of the dual complex~\eqref{eq:dual-2D-complex} (inherited from the exactness of the original complex~\eqref{eq:2D-complex}) implies that if $\dive \bsigma = 0$, which is equivalent to $D_{1}^{*} \bsigma = 0$, then, there exists a function $\varphi$ such that
\begin{equation*}
  \sigma^{ij} = (D_{2}^{*} \bA)^{ij} = \epsilon^{ik}\epsilon^{jl} \partial_{k} \partial_{l} \varphi,
\end{equation*}
or, in other words, that
\begin{equation*}
  \sigma^{11} = \partial_{y}^{2} \varphi, \qquad  \sigma^{12} = -\partial_{xy}^{2} \varphi, \qquad \sigma^{22} = \partial_{x}^{2} \varphi,
\end{equation*}
where $\varphi$ is the \emph{Airy potential}.

In dimension 3, we get
\begin{equation*}\label{eq:3D-complex}
  \Omega_{2}^{0}(\RR^{3})\
  \stackrel{\rd_{0}}{\longrightarrow} \
  \Omega_{2}^{1}(\RR^{3})\
  \stackrel{\rd_{1}}{\longrightarrow} \
  \Omega_{2}^{2}(\RR^{3}) \
  \stackrel{\rd_2}{\longrightarrow} \
  \Omega_{2}^{3}(\RR^{3}) \
  \stackrel{\rd_3}{\longrightarrow} \
  \Omega_{2}^{4}(\RR^{3})
  \stackrel{\rd_4}{\longrightarrow} \
  \Omega_{2}^{5}(\RR^{3})
  \stackrel{\rd_5}{\longrightarrow} \
  \Omega_{2}^{6}(\RR^{3})
  ,
\end{equation*}
and its dual complex is
\begin{equation*}\label{eq:dual-3D-complex}
  \Omega_{2}^{0}(\RR^{2})^*\
  \stackrel{\rd^* _{0}}{\longleftarrow} \
  \Omega_{2}^{1}(\RR^{2})^*\
  \stackrel{\rd^* _{1}}{\longleftarrow} \
  \Omega_{2}^{2}(\RR^{2})^* \
  \stackrel{\rd^* _2}{\longleftarrow}\
  \Omega_{2}^{3}(\RR^{2})^*
  \stackrel{\rd^* _3}{\longleftarrow}\
  \Omega_{2}^{4}(\RR^{2})^*
  \stackrel{\rd^* _4}{\longleftarrow }\
  \Omega_{2}^{5}(\RR^{2})^*
  \stackrel{\rd^* _5}{\longleftarrow }\
  \Omega_{2}^{6}(\RR^{2})^*.
\end{equation*}
We have now
\begin{align*}
   & \star_{4} \colon \Omega_{2}^{4}(\RR^{3})          \to \Omega^{2}_{2}(\RR^{3})^{*}, &  & W_{klmn}     \mapsto \sigma^{ij} = \epsilon^{ikm}\epsilon^{jln} W_{klmn},
  \\
   & \star_{5} \colon \Omega_{2}^{5}(\RR^{3})          \to \Omega^{1}_{1}(\RR^{3})^{*}, &  & B_{jklmn}    \mapsto \xi^{i} = \epsilon^{ikm}\epsilon^{jln} B_{jklmn},
  \\
   & \star_{2}^{-1} \colon \Omega_{2}^{4}(\RR^{3})^{*} \to \Omega_{2}^{2}(\RR^{3}),     &  & A^{ijkl}     \mapsto \varepsilon_{mn} = \frac{1}{4} \epsilon_{ikm} \epsilon_{jln} A^{ijkl},
  \\
   & \star_{4}^{-1} \colon \Omega_{2}^{2}(\RR^{3})^{*} \to \Omega_{2}^{4}(\RR^{3}),     &  & \sigma^{ij}  \mapsto  W_{klmn} = \frac{1}{4} \epsilon_{ikm}\epsilon_{jln} \sigma^{ij},
\end{align*}
and
\begin{align*}
   & (\rd_{2}\bepsilon)_{ijk} = \frac{2}{3}(\partial_{k}\varepsilon_{ij} - \partial_{i}\varepsilon_{jk}),
  \\
   & (\rd_{3}\bK)_{ijkl} = \frac{1}{4}(\partial_{l}K_{ijk} + \partial_{k}K_{jil} + \partial_{j}K_{kli} + \partial_{i}K_{lkj} ),
  \\
   & (\rd_{4}\bW)_{ijklm} = \frac{1}{2} (\partial_{i}W_{kjml} + \partial_{k}W_{mjil} + \partial_{m}W_{ijkl}).
\end{align*}
We deduce therefore that
\begin{align*}
   & \left(D_{1}^{*} \bsigma\right)^{j} = \left(\star_{5} \, \rd_{4} \, \star_{4}^{-1} \bsigma\right)^{j} = \frac{3}{2} \partial_{i} \sigma^{ij}
  \\
   & \left(D_{2}^{*} \bA\right)^{ij} = \left(\star_{4} \, \rd_{3} \, \rd_{2} \ \star^{-1}_{2} \bA\right)^{ij} = \left(\star_{4} \, \rd_{3} \, \rd_{2} \ \bphi\right)^{ij} =  \epsilon^{imk}\epsilon^{jnl} \partial_{k}\partial_{l} \phi_{mn},
\end{align*}
where $\bphi = \frac{4}{3} \star^{-1}_{2} \bA$ and where we have used the following identities
\begin{equation*}
  \epsilon^{ijk}\epsilon_{imn} = \delta^{j}_{m}\delta^{k}_{n} - \delta^{j}_{n}\delta^{k}_{m}, \qquad \epsilon^{ijk}\epsilon_{ijn} = 2\delta^{k}_{n}, \qquad \epsilon^{ijk}\epsilon_{ijk} = 6.
\end{equation*}
Therefore, in 3D, the exactness of the complex~\eqref{eq:3D-complex} implies that if $\dive \bsigma = 0$, which is equivalent to $D_{1}^{*} \bsigma = 0$, then, there exists a tensor field $\bphi \in \Omega_{2}^{2}(\RR^{3})$ such that
\begin{equation*}
  \sigma^{ij} = \epsilon^{imk}\epsilon^{jnl} \partial_{k}\partial_{l} \phi_{mn}.
\end{equation*}
The symmetric second-order tensor $\bphi$ is known in Mechanics as the \emph{Beltrami stress tensor} and its components as \emph{Beltrami stress functions} \cite{Kroe1958,Car1966,GP2004,Sad2009,GP2015}.

To conclude this section, it is important to emphasise the analogy between the diagram~\eqref{eq:Tonti-diagram-1} for electromagnetism and the following one for linear elasticity

\begin{equation}\label{eq:Tonti-diagram-2}
  \begin{array}{ccccccc}
    \bxi^{\flat} & \stackrel{\rd_{1}}{\longrightarrow}    & \bepsilon                 & \stackrel{\rd_{3}\rd_{2}}{\longrightarrow}        & \bW & \stackrel{\rd_{4}}{\longrightarrow}    & 0       \\
                 &                                        & \rotatebox{90}{$\bowtie$} &                                                   &     &                                        &
    \\
    0            & \stackrel{\rd_{1}^{*}}{\longleftarrow} & \bsigma                   & \stackrel{\rd_{2}^{*}\rd_{3}^{*}}{\longleftarrow} & \bA & \stackrel{\rd_{4}^{*}}{\longleftarrow} & \bullet
  \end{array}
\end{equation}

These kind of diagrams are known as \emph{Tonti diagrams}, see~\cite{Sou1992,DH1999,DH2002,Ton2003} and summarize the compatibility relations between some physical quantities, together with their dual quantities, the whole picture being connected by a constitutive law, here between $\bepsilon$ and $\bsigma$.

\section*{Conclusion}

We have proposed a new interpretation of the Elasticity complex using Dubois-Violette--Henneaux's theory \cite{DH1999,DH2002}, which differentiates from de Rham complex only through the definition of the exterior derivative, which takes account of the index symmetries of the tensors involved. This is why we consider this approach as natural, since it does not require to introduce \textit{ad hoc} isomorphisms, such as in the BGG approach. The fundamental property of de Rham complex, $d^2=0$ transfers directly but with $d^3=0$, rather than $d^2=0$, which explains why second-order derivatives are involved in the Saint--Venant compatibility condition, contrary to first order derivatives in the compatibility conditions of the de Rham complex. An homotopy formula has been provided which allows to recover Cesàro-Volterra path integral formula and produces also a new obstruction term measuring the degree of incompatibility for the strain. We also furnished an explanation about the number of independent components of the Saint--Venant tensor in dimension 2 and 3, and the link between the various expressions of this compatibility condition. Finally, we have provided a natural formulation of the Airy potential (2D) and the Beltrami stress potential (3D), using the dual of the Elasticity complex and the introduction of a Hodge star operator for this complex.

\appendix

\section{Generalized differential complexes}
\label{sec:Dubois-Violette}

Dubois-Violette and Henneaux \cite{DH1999, DH2002} have generalized de Rham's complex for covariant tensor fields with other index symmetries obeying certain rules. To better understand these rules, it is necessary to introduce first \emph{Young diagrams} and \emph{Young tableaux}.

A \emph{Young diagram} $Y$ represents graphically the decomposition of an integer $k$ into a partition $(k_{1}, \dotsc , k_{p})$ where $k = k_{1} + \dotsb + k_{p}$ and $k_{1} \ge \dotsb \ge k_{p} \ge 1$. For instance, the partitions
\begin{equation*}
  (1,1,1,1), \quad (2,1,1), \quad (2,2), \quad (3,1), \quad (4),
\end{equation*}
of the same integer $k=4$ correspond respectively to the following Young diagrams
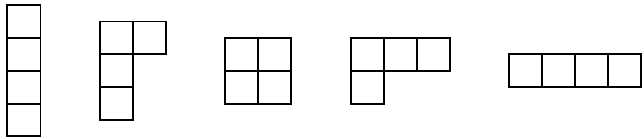
\begin{equation*}
  \yng(1,1,1,1) \qquad \yng(2,1,1) \qquad \yng(2,2) \qquad \yng(3,1) \qquad  \yng(4)
\end{equation*}

A \emph{Young tableau} $D$ is a filling of a Young diagram $Y$ by the integers $1,2,\dotsc , k$. For instance, the following tableaux
\begin{equation*}
  \young(12,3) \qquad \young(13,2) \qquad \young(21,3) \qquad \young(23,1) \qquad \young(31,2) \qquad \young(32,1)
\end{equation*}
correspond to all the possible filling of the Young diagram
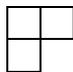
\begin{equation*}
  \yng(2,1)
\end{equation*}

To each Young tableau $D$ of size $k$, one associates a subspace $S^{D}(\RR^{n})^{*}$ of the space $\bigotimes^{k}(\RR^{n})^{*}$ of covariant tensors of order $k$, defined in
the following way. A tensor $\bT = (T_{i_{1} \dotsb i_{k}})$ belongs to $S^{D}(\RR^{n})^{*}$ if and only if:
\begin{enumerate}
  \item $\bT$ is alternate on each column of $D$. This means that if $i_{p}$ and $i_{q}$ belongs to the same column $C_{j}$ of $D$, then,
        \begin{equation*}
          T_{i_{1} \dotsb i_{q} \dotsb i_{p} \dotsb i_{n}} = - T_{i_{1} \dotsb i_{p} \dotsb i_{q} \dotsb i_{n}};
        \end{equation*}
  \item The total alternation of $\bT$ on the columns $C_{j} = \set{i_{j_{1}}, \dotsc ,i_{{j_{k}}}}$ of $D$ and an index $i_{j_{k+1}}$ in an adjacent box on the right of this column vanishes. This means that if $\mathfrak{S}_{k+1}$ denotes the permutation group of the elements $\set{i_{j_{1}}, \dotsc , i_{j_{k}},i_{j_{k+1}}}$, then,
        \begin{equation*}
          \sum_{\sigma \in \mathfrak{S}_{k+1}} \varepsilon(\sigma)\, \sigma \star \bT = 0,
        \end{equation*}
        where $\varepsilon(\sigma) = \pm 1$ is the signature of the permutation $\sigma$ and $\sigma \star \bT$ is the action of the permutation group on its components, and defined by
        \begin{equation*}
          \left(\sigma \star \bT\right)_{i_{1}\dotsb i_{n}} := T_{\sigma^{-1}(i_{1})\dotsb \sigma^{-1}(i_{n})}.
        \end{equation*}
\end{enumerate}

The dimension of the vector space $S^{D}(\RR^{n})^{*}$ is
\begin{equation*}
  \dim S^{D}\RR^{n} = \frac{\prod_{i,j}(n+j-i)}{\prod_{i,j} h_{Y}(i,j)},
\end{equation*}
where $h_{Y}(i,j)$ is the \emph{hook length} of the cell $(i,j)$ in $Y$, which is defined as 1 + the number of cells immediately at the right of $(i,j)$ + the number of cells immediately under the cell $(i,j)$.

\begin{exam}
  Let $\Omega$ in $S^{D}\RR^{n}$ where $D$ is the following Young tableau
  \begin{equation*}
    \young(1,2)
  \end{equation*}
  Then $\Omega$ has the following index symmetries
  \begin{equation*}
    \Omega_{ij} = -\Omega_{ji}.
  \end{equation*}
  These symmetries correspond to those of the spin tensor. More generally, a Young Tableau which consists of a unique column of $k$ boxes corresponds to the symmetries of an alternate tensor of degree $k$.
\end{exam}

\begin{exam}
  Let $\varepsilon$ in $S^{D}\RR^{n}$ where $D$ is the following Young tableau
  \begin{equation*}
    \young(12)
  \end{equation*}
  Then $\varepsilon$ has the following index symmetries
  \begin{equation*}
    \varepsilon_{ij} = \varepsilon_{ji}.
  \end{equation*}
  These symmetries correspond to those of the strain tensor. More generally, a Young Tableau which consists of a unique row of $k$ boxes corresponds to the symmetries of a totally symmetric tensor of order $k$.
\end{exam}

\begin{exam}
  Let $\bK$ in $S^{D}\RR^{n}$ where $D$ is the following Young tableau
  \begin{equation*}
    \young(12,3)
  \end{equation*}
  Then, $\bK$ has the following index symmetries
  \begin{align*}
     & K_{ijk} = - K_{kji},                                           \\
     & K_{ijk} + K_{jki} + K_{kij} - K_{jik} - K_{kji} - K_{ikj} = 0,
  \end{align*}
  but the later identity recasts (using the first one) as
  \begin{equation*}
    K_{ijk} + K_{jki} + K_{kij} = 0.
  \end{equation*}
\end{exam}

\begin{exam}
  Let $R$ in $S^{D}\RR^{n}$ where $D$ is the following Young tableau
  \begin{equation*}
    \young(13,24)
  \end{equation*}
  Then $R$ has the following index symmetries
  \begin{align*}
     & R_{jikl} = - R_{ijkl},               \\
     & R_{ijlk} = - R_{ijkl},               \\
     & R_{ijkl} + R_{jkil} + R_{kijl} = 0 , \\
     & R_{ijkl} + R_{iklj} + R_{iljk} = 0 .
  \end{align*}
  These symmetries correspond to those of the Riemann curvature tensor. It is known that it has the additional symmetry $R_{klij} = R_{jikl}$, which can be deduced from the preceding ones.
\end{exam}

\begin{exam}
  Let $W$ in $S^{D}\RR^{n}$ where $D$ is the following Young tableau
  \begin{equation*}
    \young(12,34)
  \end{equation*}
  Then $W$ has the following index symmetries
  \begin{align*}
     & W_{kjil} = - W_{ijkl},               \\
     & W_{ilkj} = - W_{ijkl},               \\
     & W_{ijkl} + W_{jkil} + W_{kijl} = 0 , \\
     & W_{ijkl} + W_{iklj} + W_{iljk} = 0 , \\
  \end{align*}
  and the additional symmetry deduced from the preceding ones $W_{klij} = W_{ijkl}$. These symmetries correspond to those of the Saint-Venant tensor~\eqref{eq:SV-tensor}.
\end{exam}

Given a Young tableau $D$ of size $k$, one introduces the subgroup $C$ of $\mathfrak{S}_{k}$ which preserves the columns of $D$ and the subgroup $R$, which preserves the rows of $D$.

\begin{exam}
  Let $D$ be the following Young tableau
  \begin{equation*}
    \young(13,24)
  \end{equation*}
  of size $4$. Then,
  \begin{equation*}
    C = \set{e, (12), (34), (12)(34)} \quad \text{and} \quad R = \set{e,(13), (24), (13)(24)}.
  \end{equation*}
\end{exam}

Then, one introduces the following operators on tensors $\bT$ of order $k$
\begin{equation*}
  A_D * \bT := \sum_{\tau \in C} \varepsilon(\tau)\, \tau \star \bT, \qquad S_D := \sum_{\sigma \in R} \sigma \star \bT ,
\end{equation*}
and finally
\begin{equation*}
  F_D \star \bT = \frac{1}{\mu(Y)} A_D \star (S_D \star \bT),
\end{equation*}
where $\mu(Y)$ is a normalization constant (depending on the Young diagram $Y$ encoded by $D$, rather than $D$ itself) which ensures that $F$ is a projector \cite{Ful1997}, meaning that $F^{2} = F$. It is defined as follows:
\begin{equation*}
  \mu(Y) = \prod_{i,j} h_{Y}(i,j),
\end{equation*}
where $h_{Y}(i,j)$ is the \emph{hook length} of the cell $(i,j)$ in $Y$, which is defined as 1 + the number of cells immediately at the right of $(i,j)$ + the number of cells immediately under the cell $(i,j)$.

\begin{exam}
  Consider the following Young tableau
  \begin{equation*}
    D = \young(12,3) \quad \text{with shape} \quad Y = \yng(2,1)
  \end{equation*}
  We get
  \begin{equation*}
    R = \set{e,(12)}, \quad \text{and} \quad C = \set{e, (13)},
  \end{equation*}
  and thus
  \begin{equation*}
    (S_D \star \bT)_{ijk} = T_{ijk} + T_{jik}, \quad (A_D \star \bT)_{ijk} = T_{ijk} - T_{kji}.
  \end{equation*}
  We have therefore
  \begin{equation*}
    \left(A_D \star (S_D \star \bT)\right)_{ijk} = (S_D \star \bT)_{ijk} - (S_D \star \bT)_{kji} = (T_{ijk} + T_{jik}) - (T_{kji} + T_{jki}),
  \end{equation*}
  whereas
  \begin{equation*}
    \mu(Y) = h_{Y}(1,1) \times h_{Y}(1,2) \times h_{Y}(2,1) = 3 \times 1 \times 1 = 3.
  \end{equation*}
  We get finally
  \begin{equation*}
    (F_D \star \bT)_{ijk} = \frac{1}{3} (T_{ijk} + T_{jik} - T_{kji} - T_{jki}).
  \end{equation*}
\end{exam}

\begin{exam}
  Consider the following Young tableau
  \begin{equation*}
    D = \young(12,34) \quad \text{with shape} \quad Y = \yng(2,2)
  \end{equation*}
  We get
  \begin{equation*}
    R = \set{e,(12),(34),(12)(34)}, \quad \text{and} \quad C = \set{e, (13), (24), (13)(24)},
  \end{equation*}
  and thus
  \begin{align*}
    (S_D \star \bT)_{ijkl} & = T_{ijkl} + T_{jikl} + T_{ijlk} + T_{jilk}, \\
    (A_D \star \bT)_{ijkl} & = T_{ijkl} - T_{kjil} - T_{ilkj} + T_{klij}.
  \end{align*}
  We have therefore
  \begin{align*}
    \left(A_D \star (S_D \star \bT)\right)_{ijkl} & = (S_D \star \bT)_{ijkl} - (S_D \star \bT)_{kjil} - (S_D \star \bT)_{ilkj} + (S_D \star \bT)_{klij} \\
                                                  & = T_{ijkl} + T_{jikl} + T_{ijlk} + T_{jilk} - T_{kjil} - T_{jkil} - T_{kjli} - T_{jkli}             \\
                                                  & - T_{ilkj} - T_{likj} - T_{iljk} - T_{lijk} + T_{klij} + T_{lkij} + T_{klji} + T_{lkji},
  \end{align*}
  whereas
  \begin{equation*}
    \mu(Y) = h_{Y}(1,1) \times h_{Y}(1,2) \times h_{Y}(2,1) \times h_{Y}(2,2) = 3 \times 2 \times 2 \times 1 = 12.
  \end{equation*}
\end{exam}

We shall now define the differential generalized complexes introduced by Dubois--Violette \& Henneaux in~\cite{DH1999,DH2002}. Given an integer $N \ge 1$, we define the following sequence of Young tableaux $(D_{N}^{k})$ ($k\ge 1$), starting with $k=1$ with the tableau
\begin{equation*}
  \young(1)
\end{equation*}
and where the tableau numbered $k+1$ is obtained from the tableau numbered $k$ by adding at each step, a new cell numbered $k+1$ on the first row which length does not exceed $N-1$ or otherwise on a new row at the end of the diagram, and so on.

\begin{exam}[Case $N=1$]
  \begin{equation*}
    \young(1) \qquad \young(1,2) \qquad \young(1,2,3) \qquad \dotsb
  \end{equation*}
\end{exam}

\begin{exam}[Case $N=2$]\label{ex:symmetries}
  \begin{equation*}
    \young(1) \qquad \young(12) \qquad \young(12,3) \qquad \young(12,34) \qquad \young(12,34,5) \qquad \dotsb
  \end{equation*}
\end{exam}

\begin{exam}[Case $N=3$]
  \begin{equation*}
    \young(1) \qquad \young(12) \qquad \young(123 )\qquad \young(123,4) \qquad \young(123,45) \qquad \young(123,456) \qquad \dotsb
  \end{equation*}
\end{exam}

\begin{defn}\label{def:DV-operators}
  Given $N \ge 1$, the generalized complex, introduced in ~\cite{DH1999,DH2002}, is defined as the sequence of spaces of tensor fields $(\Omega_{N}^{k}(\RR^{n}))_{k \ge 1}$, each one, having the index symmetry of the Young Tableau $(D_{N}^{k})$, and where $\Omega_{N}^{0}(\RR^{n}) = \Cinf(\RR^{n})$. The differentials of the complex
  \begin{equation*}
    {\rd^{(N)}}_{k} : \Omega_{N}^{k}(\RR^{n}) \to \Omega_{N}^{k+1}(\RR^{n}),
  \end{equation*}
  are defined by
  \begin{equation*}
    ({\rd^{(N)}}_{k} \bT)_{i_{1} \dotsc i_{k+1}} := (F_{D_{N}^{k+1}} \star \nabla \bT)_{i_{1} \dotsc i_{k+1}},
  \end{equation*}
  where $\nabla \omega$ is the gradient of the tensor field $\bT$. The differentials of the complex satisfy
  \begin{equation*}
    {\rd^{(N)}}_{k+N}\, \dotsb \, {\rd^{(N)}}_{k} = 0, \qquad \forall k \ge 0,
  \end{equation*}
  which we shall summarize by the equation $\rd^{N+1} = 0$.
\end{defn}

\begin{rem}
  For $N=1$, we recover the de Rham complex and $ \rd_{k} = \rd$ is just the exterior derivative up to a scaling factor. Using this graphical notation, the de Rham complex on $\RR^{n}$ can be recast as
  \begin{equation*}
    \young(1) \overset{\rd_{1}}{\longrightarrow} \young(1,2) \overset{\rd_{2}}{\longrightarrow} \young(1,2,3) \overset{\rd_{3}}{\longrightarrow} \dotsb
  \end{equation*}
  where $\rd_{k} = (1/k!)\, d$.
\end{rem}

\begin{exam}\label{exam:DV-derivatives}
  When $N=2$, and we have set ${\rd^{(2)}}_{k} = \rd_{k}$, we obtain the following formulas.
  \begin{itemize}
    \item If $\bxi^\flat \in \Omega_{2}^{1}(\RR^{n})$, we get
          \begin{equation*}
            (\rd_{1}\bxi^\flat)_{ij} = \frac{1}{2}(\partial_{i}\xi_{j} + \partial_{j}\xi_{i}).
          \end{equation*}
    \item If $\bepsilon \in \Omega_{2}^{2}(\RR^{n})$, we get
          \begin{equation*}
            (\rd_{2}\bepsilon)_{ijk} = \frac{2}{3}(\partial_{k}\varepsilon_{ij} - \partial_{i}\varepsilon_{jk}).
          \end{equation*}
    \item If $\bK \in \Omega_{2}^{3}(\RR^{n})$, we get
          \begin{equation*}
            (\rd_{3}\bK)_{ijkl} = \frac{1}{4}(\partial_{l}K_{ijk} + \partial_{k}K_{jil} + \partial_{j}K_{kli} + \partial_{i}K_{lkj} ).
          \end{equation*}
    \item If $\bW \in \Omega_{2}^{4}(\RR^{n})$, we get
          \begin{equation*}
            (\rd_{4}\bW)_{ijklm} = \frac{1}{2} (\partial_{i}W_{kjml} + \partial_{k}W_{mjil} + \partial_{m}W_{ijkl}).
          \end{equation*}
  \end{itemize}
\end{exam}

For the de Rham complex, the \emph{formal adjoint} or \emph{co-differential} $\rd^{*}$ of $\rd$ is defined using the \emph{Hodge star operator}, defined as
\begin{equation*}
  \star_{k} \colon \Omega^{k}(\RR^{n}) \to \Omega^{n-k}(\RR^{n})^{*}, \qquad \omega_{i_{1} \dotsb i_{k}} \mapsto \epsilon^{j_{1} \dotsb j_{n-k}i_{1} \dotsb i_{k}}\omega_{i_{1} \dotsb i_{k}},
\end{equation*}
where $\epsilon^{i_{1} \dotsb i_{n}}$ is the (contravariant) Levi-Civita symbol. The fundamental observation~\cite{DH1999} is that this operator can be extended to generalized complexes. In particular, for $N=2$, we get the \emph{generalized Hodge operator}
\begin{equation*}
  \star_{k} \colon \Omega^{k}_{2}(\RR^{n})\longrightarrow\Omega^{2n-k}_{2}(\RR^{n})^{*},
\end{equation*}
defined by
\begin{equation}\label{def:generalised-Hodge}
  \left(\star_{k}\bT\right)^{j_{1}\cdots j_{2n-k}} :=
  \left\{
  \begin{array}{ll}
    \epsilon^{i_{1} i_{3} \cdots i_{k-1} j_{1} j_{3} \cdots j_{2n-k-1}} \epsilon^{i_{2} i_{4} \cdots i_{k} j_{2} j_{4} \cdots j_{2n-k}} T_{i_{1}\cdots i_{k}}, & \hbox{if $k$ is even;} \\
    \epsilon^{i_{1} i_{3} \cdots i_{k} j_{1} j_{3} \cdots j_{2n-k}} \epsilon^{i_{2} i_{4} \cdots i_{k-1} j_{2} j_{4} \cdots j_{2n-k-1}} T_{i_{1}\cdots i_{k}}, & \hbox{if $k$ is odd.}
  \end{array}
  \right.
\end{equation}

This linear operator is invertible and allows to define the \emph{generalized co-differential}
\begin{equation*}
  \rd^{\star}_{k} \colon \Omega_{2}^{k+1}(\RR^{n})^{*} \longrightarrow \Omega_{2}^{k}(\RR^{n})^{*},
\end{equation*}
given by
\begin{equation}\label{eq:generalized-d-star}
  \rd^{\star}_{k} := \star_{2n-k} \, \rd_{2n-k-1} \, {\star_{2n-k-1}}^{-1},
\end{equation}
and which satisfies
\begin{equation*}
  \rd^{\star}_{k} \, \rd^{\star}_{k+1} \, \rd^{\star}_{k+2} = 0.
\end{equation*}

\begin{rem}
  Since the generalized Dubois-Violette--Henneaux complex is exact, its dual complex is exact also. Indeed, an homotopy formula such as
  \begin{equation*}
    d_{2n-k-1} K_{2n-k-1} +  K_{2n-k}d_{2n-k} = \id, \qquad \text{on $\Omega_{2}^{2n-k}(\RR^{n})$}
  \end{equation*}
  induces an homotopy formula for the dual complex as
  \begin{multline*}
    \left(\star_{2n-k} \, \rd_{2n-k-1} \, {\star_{2n-k-1}}^{-1}\right) \left(\star_{2n-k-1} \, K_{2n-k-1} \, {\star_{2n-k}}^{-1}\right)
    \\
    + \left(\star_{2n-k} \,  K_{2n-k} \, {\star_{2n-k+1}}^{-1}\right) \left(\star_{2n-k+1} \, \rd_{2n-k} \, {\star_{2n-k}}^{-1}\right) = \id,
  \end{multline*}
  or in a more readable expression as
  \begin{equation*}
    d^{*}_{k} \,  K^{*}_{k} +  K^{*}_{k-1} d^{*}_{k-1} = \id, \qquad \text{on  $\Omega_{2}^{k}(\RR^{n})^{*}$},
  \end{equation*}
  where
  \begin{equation*}
    K^{*}_{k} := \star_{2n-k-1} \, K_{2n-k-1} \, \star^{-1}_{2n-k} \colon \Omega_{2}^{k}(\RR^{n})^{*} \to \Omega_{2}^{k+1}(\RR^{n})^{*}.
  \end{equation*}
\end{rem}


\begin{thebibliography}{10}

\bibitem{Air1863}
G.~B. Airy.
\newblock On the strains in the interior of beams.
\newblock {\em Philosophical Transactions of the Royal Society of London},
  153:49--79, Dec. 1863.

\bibitem{AAA2023}
R.~Aloev, I.~Abdullah, A.~Akbarova, and S.~Juraev.
\newblock Lyapunov stability of the numerical solution of the {S}aint-{V}enant
  equation.
\newblock In {\em {AIP} Conference Proceedings}, volume 2484. {AIP} Publishing.
\newblock Issue: 1.

\bibitem{ACG2006}
C.~Amrouche, P.~G. Ciarlet, L.~Gratie, and S.~Kesavan.
\newblock On {Saint} {Venant's} compatibility conditions and {Poincar\'e's}
  lemma.
\newblock {\em Comptes Rendus. Math\'ematique}, 342(11):887--891, 2006.

\bibitem{AFW2006}
D.~N. Arnold, R.~S. Falk, and R.~Winther.
\newblock Differential complexes and stability of finite element methods ii:
  The elasticity complex.
\newblock In D.~N. Arnold, P.~B. Bochev, R.~B. Lehoucq, R.~A. Nicolaides, and
  M.~Shashkov, editors, {\em Compatible Spatial Discretizations}, pages 47--67,
  New York, NY, 2006. Springer New York.

\bibitem{Arnold2007}
D.~N. Arnold, R.~S. Falk, and R.~Winther.
\newblock Mixed finite element methods for linear elasticity with weakly
  imposed symmetry.
\newblock {\em Mathematics of Computation}, 76(260):1699–1724, Oct. 2007.

\bibitem{CH2024}
A.~Cap and K.~Hu.
\newblock {BGG} sequences with weak regularity and applications.
\newblock {\em Found Comput Math 24}, page 1145–1184, 2024.

\bibitem{Car1966}
D.~Carlson.
\newblock On the completeness of the {{B}eltrami} stress functions in continuum
  mechanics.
\newblock {\em Journal of Mathematical Analysis and Applications},
  15(2):311--315, Aug. 1966.

\bibitem{Ces1906}
E.~Cesàro.
\newblock Sulle formole del volterra fondamentali nella teoria delle
  distorsioni elastiche.
\newblock {\em Il Nuovo Cimento Series 5}, 12(1):143 – 154, 1906.
\newblock Cited by: 11.

\bibitem{SKE2020}
S.~H. Christiansen, K.~Hu, and E.~Sande.
\newblock Poincaré path integrals for elasticity.
\newblock {\em Journal de Mathématiques Pures et Appliquées}, 135:83--102,
  2020.

\bibitem{CCG2007}
P.~G. Ciarlet, P.~Ciarlet, G.~Geymonat, and F.~Krasucki.
\newblock Characterization of the kernel of the operator {CURL\,CURL}.
\newblock {\em Comptes Rendus. Math\'ematique}, 344(5):305--308, 2007.

\bibitem{CGM2010}
P.~G. Ciarlet, L.~Gratie, and C.~Mardare.
\newblock A {C}esàro–{V}olterra formula with little regularity.
\newblock {\em Journal de Mathématiques Pures et Appliquées}, 93(1):41--60,
  2010.

\bibitem{dSai1855}
B.~de~{S}aint {V}enant.
\newblock {\em De la torsion des prismes}.
\newblock Imprimerie impériale, Paris, 1855.
\newblock Extrait du tome XIV des mémoires présentés par divers savants à
  l'académie des sciences.

\bibitem{DH1999}
M.~Dubois-Violette and M.~Henneaux.
\newblock Generalized cohomology for irreducible tensor fields of mixed {Y}oung
  symmetry type.
\newblock {\em Lett. Math. Phys.}, 49(3):245--252, 1999.

\bibitem{DH2002}
M.~Dubois-Violette and M.~Henneaux.
\newblock Tensor fields of mixed {Y}oung symmetry type and n-complexes.
\newblock {\em Comm. Math. Phys.}, 226(2):393--418, 2002.

\bibitem{Eas1999}
M.~Eastwood.
\newblock Variations on the de {R}ham complex.
\newblock {\em Notices AMS}, 46:1368--1376, 1999.

\bibitem{Eas2000}
M.~Eastwood.
\newblock A complex from linear elasticity.
\newblock In {\em Proceedings of the 19th Winter School" Geometry and
  Physics"}, pages 23--29. Circolo Matematico di Palermo, 2000.

\bibitem{Fal2008}
R.~S. Falk.
\newblock {\em Finite Element Methods for Linear Elasticity}, pages 159--194.
\newblock Springer Berlin Heidelberg, Berlin, Heidelberg, 2008.

\bibitem{Ful1997}
W.~Fulton.
\newblock {\em {Y}oung tableaux}, volume~35 of {\em London Mathematical Society
  Student Texts}.
\newblock Cambridge University Press, Cambridge, 1997.
\newblock With applications to representation theory and geometry.

\bibitem{GHL2004}
S.~Gallot, D.~Hulin, and J.~Lafontaine.
\newblock {\em Riemannian Geometry}.
\newblock Universitext. Springer Berlin Heidelberg, Berlin, third edition,
  2004.

\bibitem{Gar2022}
J.~Garrigues.
\newblock Algèbre et analyse tensorielles pour l’étude des milieux
  continus.
\newblock Centrale Marseille, Oct. 2022.
\newblock available at: https://cel.hal.science/cel-00679923.

\bibitem{GP2015}
D.~V. Georgievskii and B.~E. Pobedrya.
\newblock On the compatibility equations in terms of stresses in
  many-dimensional elastic medium.
\newblock {\em Russian Journal of Mathematical Physics}, 22(1):6--8, Jan. 2015.

\bibitem{GP2004}
D.~Georgiyevskii and B.~Pobedrya.
\newblock The number of independent compatibility equations in the mechanics of
  deformable solids.
\newblock {\em Journal of Applied Mathematics and Mechanics}, 68(6):941--946,
  2004.

\bibitem{Gur1963}
M.~E. Gurtin.
\newblock A generalization of the {B}eltrami stress functions in continuum
  mechanics.
\newblock {\em Archive for Rational Mechanics and Analysis}, 13(1):321--329,
  Dec. 1963.

\bibitem{Hu2024}
K.~Hu.
\newblock {\em Nonlinear Elasticity Complex and a Finite Element Diagram
  Chase}, pages 231--252.
\newblock Springer Nature Singapore, 2024.

\bibitem{Kroe1958}
E.~Kröner.
\newblock {\em Kontinuumstheorie der Versetzungen und Eigenspannungen}.
\newblock Springer, 1958.

\bibitem{Lan1999}
S.~Lang.
\newblock {\em Fundamentals of Differential Geometry}, volume 191 of {\em
  Graduate Texts in Mathematics}.
\newblock Springer-Verlag, New York, 1999.

\bibitem{LS1954}
H.~Langhaar and M.~Stippes.
\newblock Three-dimensional stress functions.
\newblock {\em Journal of the Franklin Institute}, 258(5):371--382, Nov. 1954.

\bibitem{Lov1944}
A.~E.~H. Love.
\newblock {\em A Treatise on the Mathematical Theory of Elasticity}.
\newblock Dover, New York, 4th edition, 1944.

\bibitem{Mus1977}
N.~I. Muskhelishvili.
\newblock {\em Some Basic Problems of the Mathematical Theory of Elasticity}.
\newblock Springer Netherlands, 1977.

\bibitem{CGM2007}
C.~M. P.G.~Ciarlet, L.~Gratie.
\newblock Intrinsic methods in elasticity: a mathematical survey.
\newblock {\em Discrete Contin. Dyn. Syst}, 2007.

\bibitem{CGM2009}
C.~M. Philippe G. Ciarlet~a, Liliana Gratie~a.
\newblock A generalization of the classical {C}esàro–{V}olterra path
  integral formula.
\newblock {\em C. R. Acad. Sci. Paris}, (1):347, 2009.

\bibitem{Pom2016}
J.-F. Pommaret.
\newblock Airy, {B}eltrami, {M}axwell, {E}instein and {L}anczos potentials
  revisited.
\newblock {\em Journal of Modern Physics}, 07(07):699--728, 2016.

\bibitem{Sad2009}
M.~H. Sadd.
\newblock {\em Elasticity}.
\newblock Elsevier, Academic Press, Amsterdam, 2. ed., [repr.] edition, 2009.

\bibitem{Sal2005}
J.~Salen{ç}on.
\newblock {\em Mécanique des milieux continus, Tome 1 - Concepts
  généraux}.
\newblock Editions de l{\textquoteright}Ecole polytechnique, 2005.

\bibitem{Sch1953}
H.~Schaefer.
\newblock The stress functions of the three-dimensional continuumand elastic
  bodies.
\newblock {\em ZAMM - Journal of Applied Mathematics and Mechanics /
  Zeitschrift f{ü}r Angewandte Mathematik und Mechanik},
  33(10–11):356--362, Jan. 1953.

\bibitem{Sou1992}
J.~M. Souriau.
\newblock Mécanique des états condensés de la matière.
\newblock Allocution pour le premier séminaire International de la
  fédération de mécanique de Grenoble, 19-21 mai 1992., 1992.

\bibitem{Spe1969}
D.~C. Spencer.
\newblock Overdetermined systems of linear partial differential equations.
\newblock {\em Bulletin of the American Mathematical Society}, 75(2):179--239,
  1969.

\bibitem{Teo1972}
P.~P. Teodorescu.
\newblock Stress functions in three-dimensional elastodynamics.
\newblock {\em Acta Mechanica}, 14(2–3):103--118, June 1972.

\bibitem{TGA1970}
S.~P. Timoshenko, J.~N. Goodier, and H.~N. Abramson.
\newblock Theory of elasticity (3rd ed.).
\newblock {\em Journal of Applied Mechanics}, 37(3):888--888, Sept. 1970.

\bibitem{Tin1996}
T.~T.~C. Ting.
\newblock {\em Anisotropic Elasticity: Theory and Applications}.
\newblock Oxford University Press, 04 1996.

\bibitem{Ton2003}
E.~Tonti.
\newblock A classification diagram for physical variables.
\newblock 2003.

\bibitem{Vol1906}
V.~{V}olterra.
\newblock Sur l'équilibre des corps élastiques multiplement connexes.
\newblock {\em Annales scientifiques de l'Ecole Normale Supérieure},
  24(3):401--517, 1907.

\bibitem{VV1960}
V.~{V}olterra and E.~{V}olterra.
\newblock {\em Sur les distorsions des corps élastiques (théorie et
  applications)}.
\newblock Gauthier-Villars, 1960.

\bibitem{Y2013}
A.~Yavari.
\newblock Compatibility equations of nonlinear elasticity for
  non-simply-connected bodies.
\newblock {\em Arch Rational Mech Anal}, 209(1):237–253, 2013.

\end{thebibliography}
\end{document}